\documentclass{amsart}
\calclayout
\usepackage{amsmath,amssymb,enumitem,amsfonts,microtype,mathtools,mathrsfs}
\usepackage[noadjust]{cite}
\usepackage{xcolor,times}
\usepackage[hidelinks=true]{hyperref}

\theoremstyle{plain}
\newtheorem{theorem}{Theorem}[section]

\newtheorem{proposition}[theorem]{Proposition}
\newtheorem{lemma}[theorem]{Lemma}
\newtheorem{corollary}[theorem]{Corollary}

\theoremstyle{definition}

\newtheorem{definition}[theorem]{Definition}
\newtheorem{remark}[theorem]{Remark}

\newtheorem*{notation}{Notation}

\newcommand{\annotation}[1]{\raise.8ex\hbox to0pt{\hss\ensuremath{\curlyvee}\hss}\marginpar{\tiny\baselineskip6pt #1}}

\newcommand{\occult}[1]{}

\usepackage{soul}

\newcommand{\ST}{\mathbin{\left\bracevert\text{$\phantom{|}$}\right.}}
\newcommand{\NW}{\textsl{NW}}
\newcommand{\lto}[1]{\xrightarrow[\;#1\;]{}}
\newcommand{\dfn}{\mathbin{{\vcentcolon}\hskip-1pt{=}}}
\let\emptyset\varnothing
\newcommand\R{\ensuremath{\mathbb R}}
\newcommand\Z{\ensuremath{\mathbb Z}}

\newcommand\N{\ensuremath{\mathbb N}}
\newcommand\Oc{\ensuremath{\mathcal O}}
\newcommand{\rest}[1]{{}_{\displaystyle\restriction_{\scriptstyle{#1}}}}

\begin{document}

\title[Centralizers for hyperbolic and expansive flows]{Centralizers of hyperbolic and kinematic-expansive flows}

\begin{abstract}
We show generic \(C^\infty\) hyperbolic flows (Axiom A and no cycles, but not transitive Anosov) commute with no \(C^\infty\)-diffeomorphism other than a time-$t$ map of the flow itself. Kinematic expansivity, a substantial weakening of expansivity, implies that \(C^0\) flows have quasi-discrete \(C^0\)-centralizer, and additional conditions broader than transitivity then give discrete \(C^0\)-centralizer. We also prove centralizer-rigidity: a diffeomorphism commuting with a generic hyperbolic flow is determined by its values on any open set.
\end{abstract}
\author{Lennard Bakker}
\address{L.~Bakker, Department of Mathematics, Brigham Young University, Provo, UT 84602, USA}\email{bakker@math.byu.edu}
\author{Todd Fisher}
\address{T.~Fisher, Department of Mathematics, Brigham Young University, Provo, UT 84602, USA}\email{tfisher@math.byu.edu}
\author{Boris Hasselblatt}
\address{B.~Hasselblatt, Department of Mathematics, Tufts University, Medford, MA 02144, USA}\email{boris.hasselblatt@tufts.edu}

\thanks{T.F.\ is supported by Simons Foundation grant \# 580527.}

\subjclass[2010]{37D20, 37C10, 37C20}

\maketitle

\tableofcontents

\section{Introduction}
It is natural to expect a dynamical system to have no symmetries unless it is quite special. Symmetries correspond to the existence of a diffeomorphism or flow that commutes with the given dynamics, so one expects flows to ``typically'' have small centralizers, i.e., to commute with few flows or diffeomorphisms. Our topological results imply that expansive continuous flows have essentially discrete centralizer; the natural condition for this is a weakening of expansivity that does not allow reparameterizations and hence requires only a ``kinematic'' or dynamical separation of orbits rather than a ``geometric'' one.

Smale's list of problems for the next century \cite{Smale98} includes questions regarding how typically centralizers are trivial. Given his interest in classifying dynamical systems up to conjugacy, the centralizer also describes the exact extent of non\-uniqueness of a conjugacy between dynamical systems. Our main result is that hyperbolic flows generically commute with no diffeomorphism.
\begin{definition}
Two diffeomorphisms \(f,g\) are said to \emph{commute} if \(f\circ g=g\circ f\).
A \(C^r\) flow \(\Phi\) is a 1-parameter group of \(C^r\) diffeomorphisms \(t\mapsto\varphi^t\) for \(t\in\R\)
on a closed \(C^r\)-manifold \(M\) with \(0< r\le\infty\). A \(C^0\) flow is an \(\R\)-action by homeomorphisms of a compact topological space \(X\) (and not assumed to be generated by a vector field). For \(r\ge0\) a \(C^r\) diffeomorphism $f\colon M\to M$ \emph{commutes with \(\Phi\)} if\/ $f\circ\varphi^t=\varphi^t\circ f$ for all\/ $t\in\R$.  We denote the set of such diffeomorphisms by $\mathcal{C}^r(\Phi)$ and say that \(\Phi\) has trivial \(C^r\)-centralizer if \(\mathcal{C}^r(\Phi)=\{\varphi^t\mid t\in\R\}\)\rlap.\footnote{Triviality or discreteness of a centralizer of a flow \(\Phi\) means triviality of the closed centralizer subgroup modulo its closed normal subgroup \(\Phi\).}

\(C^r\) flows \(\Phi,\Psi\) commute if all \(\phi^t,\psi^s\) do, so  $\psi^s\circ\varphi^t=\varphi^t\circ\psi^s$ for all\/ $s,t\in\R$. The set of flows \(\Psi\) that commute with \(\Phi\) is called the \emph{\(C^r\)-centralizer} of \(\Phi\), denoted $Z^r(\Phi)$. We say that $Z^r(\Phi)$ is trivial (or that \(\Phi\) has trivial \R-centralizer) if\/ $Z^r(\Phi)$ consists of all constant-time reparameterizations of \(\Phi\).  In other words,  $\Psi\in Z^r(\Phi)\Rightarrow\psi^t=\varphi^{ct}$ for some $c\in\R$ and all\/ $t\in\R$.
\end{definition}
If a flow has nontrivial centralizer, then it is part of an \(\R^k\)-action, and these exhibit some rigidity phenomena \cite[Corollary 5]{KatokSpatzier}, \cite{KS94, KN2011, RHZ14}. Hyperbolicity entails a complicated and tightly interwoven structure on the phase space that is both topologically rigid and smoothly unclassifiable. Therefore, this is a natural context in which to expect centralizers to be generically trivial, and there are a number of prior results to that effect.

\subsection{Discrete-time centralizers: commuting diffeomorphisms}
Extending results of Anderson \cite{Anderson}, Palis \cite{Palis78} proved that among $C^\infty$ Axiom A diffeomorphisms with strong transversality there is an open and dense set with discrete centralizer. Palis and Yoccoz \cite{PY89} extended this: a large class of Axiom A diffeomorphisms with strong transversality has trivial centralizer.  In \cite{Fis1} this was shown to hold for generic (non-Anosov) Axiom A diffeomorphisms with the no-cycles condition (which is weaker than strong transversality).
Rocha and Varandas \cite{RochaVarandas} have shown that the centralizer of \(C^r\)-generic diffeomorphisms at hyperbolic basic sets is trivial.
In the \(C^1\)-topology, Bonatti, Crovisier, and Wilkinson proved that diffeomorphisms generically have trivial centralizer \cite{BCWGenTriv,BCWGenTrivAnn,BCWGenTrivCons}
but jointly with Vago \cite{BCVW} they found that on any compact manifold there is a nonempty open set of \(C^1\)-diffeomorphisms with a \(C^1\)-dense subset of \(C^\infty\)-diffeomorphisms whose \(C^\infty\)-centralizer is uncountable, hence nontrivial.
Our results instead concern flows.
\subsection{Continuous-time centralizers: commuting flows}
For commuting flows there have been fewer results.  Sad \cite{Sad} proved that there is an open and dense set of\/ $C^\infty$ Axiom-A flows with strong transversality that have trivial \R-centralizer. Recently, Bonomo, Rocha, and Varandas \cite{BRV} proved that the \R-centralizer of any Komuro-expansive\footnote{Komuro-expansivity is stronger than the usual notion of expansivity and equivalent when there are no fixed points; sometimes it is called expansivity as well \cite[Section 2.1.2]{BRV}, \cite{Artigue}.} flow with non\-resonant singularities is trivial. Important classes of geometric Lorenz attractors are Komuro-expansive as this form of expansivity is more compatible with the coexistence of regular and singular orbits for the flow. These results were extended by Bonomo and Varandas \cite{BonomoVarandas} to show triviality of the centralizer at homoclinic class. Recently, Leguil, Obata, and Santiago \cite{LOS} investigated when an \R-centralizer for a flow on a manifold is not necessarily trivial, but is ``small'' in a certain sense and proved two criteria that establish this, and  Obata \cite{Obata} recently expanded these results to centralizers for generic vector fields. Our results in the realm of topological dynamics implement his suggestion that a much weaker notion than expansivity (being \emph{kinematic-expansive}) implies quasidiscrete centralizer.

We first determine $Z^r(\Phi)$ for a transitive kinematic-expansive flow $\Phi$.  Although there exist open sets of Anosov diffeomorphisms with trivial centralizer \cite{BF}, there are many examples of Anosov diffeomorphisms with nontrivial centralizer.  This is not the case for Anosov flows.  It has been said to be ``well-known and elementary'' that Anosov flows have trivial \R-centralizer\footnote{``il est bien connu (et \'el\'ementaire) qu'un champ de vecteurs qui commute avec un champ d'Anosov est n\'ecessairement un multiple constant de ce champ''  \cite[p.\ 262]{GhysFoliations}} \cite[Corollary 10.1.4]{FH}.  We extend this to kinematic-expansive flows \emph{that have no differentiability}---noting that results about \(C^0\)-centralizers seem to be exceedingly rare:\footnote{The results in \cite{Oka} are similar to our results on centralizers of kinematic-expansive flows, but not the same in that Oka defines a so-called unstable centralizer and studies it for Bowen--Walters expansive flows.}
\begin{theorem}\label{t.expansive}\label{THMExpansiveChainComp}
Transitive kinematic-expansive flows (Definition \ref{DEFtransitive}) have trivial \(C^0\) \R-centralizer.

Indeed, a kinematic-expansive flow \(\Phi\) has \emph{collinear} and quasitrivial \(C^0\) \(\R\)-centralizer \cite[Definitions 2.1, 2.6]{LOS}, i.e., a commuting flow \(\Psi\) is of the form \(\psi^t(\cdot)=\varphi^{t\cdot\tau(\cdot)}(\cdot)\) for a continuous \(\Phi\)-invariant \(\tau\). If the space is connected and \(\Phi\) has at most countably many chain-components, all of which are topologically transitive, then \(\Phi\) has trivial \(C^0\) \R-centralizer.
\end{theorem}

An application of the above result is for Axiom A flows that are quasitransverse, where an Axiom A flow is \emph{quasi\-transverse} if\/ $T_x W^u(x)\cap T_x W^s(x)=\{0\}$ for all\/ $x\in M$.  These flows are expansive (hence kinematic-expansive) and have finitely many chain components, each transitive.  Note that Anosov flows (transitive or not) are quasi\-transverse Axiom A flows.

\begin{corollary}\label{CORQuasiCentralizer}
A quasi\-transverse Axiom A flow has trivial \(C^0\) \R-centralizer.
\end{corollary}

The ``elementary'' reason for this is that commuting flows act ``isometrically'' on each other's orbits, and this is incompatible with expansivity (or hyperbolicity) unless the orbits coincide; a less elementary and more explicit argument invokes uniqueness in structural stability (Theorem \ref{thmstrucstabilityhypsets}). In either approach it remains to make the orbitwise ``isometries'' coherent (see page \pageref{PROPCentralizerCloseToIdentity}). We do this in the \(C^0\)-category where neither approach is viable. We emphasize again that there is not exactly an abundance of results about \(C^0\)-centralizers because the constraints from differentiability make these issues much more manageable (but see e.g.,  \cite{RochaC0}).

\subsection{Diffeomorphisms commuting with flows}
Returning to the point of view that we are studying the symmetry group of a flow leads us to the core question of which diffeomorphisms or homeomorphisms (rather than flows) commute with a flow, and thereby to our main results. We first note two underlying facts for \emph{continuous} flows. Refining an argument by Walters shows that the symmetry group of sufficiently ``intricate'' flows is essentially discrete, but unlike his result, ours uses a less restrictive version of expansivity called kinematic expansivity; this is done in parts \eqref{itemFixOrbits}, \eqref{itemOrbitTranslation} \& \eqref{itemCtsShift} of Proposition \ref{PROPCentralizerCloseToIdentity} below and gives:
\begin{theorem}\label{THMWaltersCtsTime}
Kinematic-expansive flows have quasidiscrete centralizer: if \(\Phi\) is a kinematic-expansive (Definition \ref{d.expansive}) continuous flow on \(X\), \(\epsilon>0\), \(\delta>0\) a separation constant for \(\epsilon\), \(f\circ\varphi^t=\varphi^t\circ f\), \(d_{C^0}(f,\mathrm{Id})<\delta\), then there is a continuous \(\Phi\)-invariant \(T \colon X\to\R\) such that $f(\cdot)=\varphi^{T(\cdot)}(\cdot)$.
\end{theorem}
Theorem \ref{THMWaltersCtsTime} is meant to be an indication of what we prove. We produce discreteness of the \(C^0\)-centralizer in rather greater generality (Propositions \ref{PROPCentralizerCloseToIdentity} and  \ref{PRPLocTrivCent} and Remark \ref{REMConstOfMotion}). This has the interesting application, in Theorem \ref{THMConjugacyUnique}, that any topological conjugacy to a transitive kinematic-expansive flow (say) is locally unique (unique when chosen near the identity). 

In this generality, one should not  expect trivial centralizer: the geodesic flow on the usual genus-2 surface has a finite symmetry group generated by reflection and rotation isometries of the double torus, and the suspension of \(\begin{psmallmatrix}2&1\\1&1\end{psmallmatrix}\) has the symmetry coming from \(x\mapsto-x\). These are expansive transitive flows and hence their \(C^1\)-perturbations have the same property and nontrivial \(C^0\) centralizer as well. By contrast,
our main results say that generic flows of the following kinds commute with no \emph{diffeomorphisms} other than time-$t$ maps of the flow.
\begin{notation}
Following \cite{PY89}, let $\mathcal{A}^r(M)$ be the set of\/ $C^r$ Axiom-A flows on a manifold $M$ that have no cycles and are not transitive Anosov,
and $\mathcal{A}^r_1(M)$ be the set of\/ $\Phi\in\mathcal{A}^r(M)$ with a fixed or periodic sink or source (i.e.,  periodic attractor or repeller).
\end{notation}
Specifically, for the latter class, indeed an open dense set of such flows has trivial centralizer:
\begin{theorem}\label{t.periodictrivial}
For a \(C^1\)-open and \(C^\infty\)-dense set
$\mathcal{O}$ of \(\Phi\in\mathcal{A}^\infty_1(M)\) we have \(\mathcal{C}^\infty(\Phi)=\{\varphi^t\mid t\in\R\}\), that is, if\/ $f\in \mathcal{C}^\infty(\Phi)$, then $f=\varphi^t$ for some $t\in\R$.
\end{theorem}

The perturbations performed are done on the wandering points, and this is why our construction does not work for transitive Anosov flows.  More specifically, while it is easy to force a commuting diffeomorphism to send each periodic orbit to itself, there is no control over the way each orbit is shifted, and in the presence of recurrence this creates problems.

Without assuming the presence of a fixed or periodic sink or source this conclusion remains true in low-dimensional situations; in full generality we obtain trivial centralizer for generic such flows.
\begin{theorem}\label{t.residualtrivial}
\(\mathcal{C}^\infty(\Phi)=\{\varphi^t\mid t\in\R\}\) for  a residual set $\mathcal{R}$ of \(\Phi\in\mathcal{A}^\infty(M)\).
If\/ $\dim(M)=3$, then this holds for an open and dense set $\mathcal{U}\subset\mathcal{A}^\infty(M)$.
\end{theorem}
One of our auxiliary results was conjectured in \cite[p.\ 83]{PY89} for discrete time and is of independent interest:
there is an open dense set of\/ $\Phi$ in $\mathcal{A}^\infty(M)$ with centralizer-rigidity.
\begin{theorem}[Rigidity]\label{THMopensetscentralizers}
If \(\Phi\in\mathcal{V}\subset\mathcal{A}^\infty(M)\) as in Proposition \ref{prop.perturbationconnecting}, $f_1, f_2\in \mathcal{C}^\infty(\Phi)$, and $f_1=f_2$ on a nonempty open set, then $f_1=f_2$ on $M$.
\end{theorem}

\noindent{\bf Acknowledgement:} We would like to express our gratitude to Davi Obata who called our attention to the possibility of weakening expansivity as a hypothesis of our results in topological dynamics.
\section{Background}
We review some basic notions pertinent to hyperbolic sets and expansivity and introduce kinematic expansivity.
\begin{definition}[Nonwandering, transitivity, expansivity \cite{BowenWalters,FH}]\label{d.expansive}\label{DEFtransitive}
A point \(x\) is \emph{nonwandering} for a flow \(\Phi\) on \(X\) if \(x\in
U\) open, $T>0\Rightarrow\exists t>T$ with $\varphi^t(U)\cap U\neq\emptyset$; otherwise it is said to be
\emph{wandering}. The (closed) set of nonwandering points is denoted by
$\NW(\Phi)$.
The $\omega$-\emph{limit set} of\/ $x\in M$ is \(\omega(x)\dfn\bigcap_{t\ge0}\overline{\varphi^{[t,\infty)}(x)}\subset\NW(\Phi)\) and the $\alpha$-\emph{limit set} is
\(\alpha(x)\dfn\bigcap_{t\le0}\overline{\varphi^{(-\infty,t]}(x)}\subset\NW(\Phi)\).
The \emph{limit set} of \(\Phi\) is \(L(\Phi)\dfn\overline{\bigcup_{x\in X}\alpha(x)\cup\omega(x)}\). \(\Phi\) is said to be (topologically) \emph{transitive} if there is a dense forward semiorbit \(\varphi^{[0,\infty)}(x)\).

\(\Phi\) is \emph{expansive}\index{expansive!flow|textbf} (or Bowen--Walters expansive for emphasis) if for all\/ $\epsilon>0$ there is a $\delta>0$, called an \emph{expansivity constant}
(for $\epsilon$),
such that if $x,y\in X$, $s\colon\R\to\R$ continuous, $s(0)=0$, and $d(\varphi^t(x),\varphi^{s(t)}(y))<\delta\ \forall t\in\R$, then $y=\varphi^t(x)$ for some $|t|<\epsilon$.

It is \emph{kinematic-expansive}\index{separating!flow|textbf} if for all\/ $\epsilon>0$ there is a $\delta>0$, called an \emph{separation constant}
(for $\epsilon$),
such that if $x,y\in X$ and $d(\varphi^t(x),\varphi^t(y))<\delta\ \forall t\in\R$,
then $y=\varphi^t(x)$ for some $|t|<\epsilon$. (Since this does not allow reparameterizations, it requires only a ``kinematic'' separation of orbits rather than a ``geometric'' one.)

It is \emph{separating}\index{separating!flow|textbf} \cite[Definition 2.3]{LOS} if there is a $\delta>0$
such that if $x,y\in X$ and $d(\varphi^t(x),\varphi^t(y))<\delta\ \forall t\in\R$, then $y\in\varphi^\R(x)$.
\end{definition}
\begin{remark}
These three expansivity are progressively less restrictive (see also \cite{Artigue,Gura1,Gura2}), and for our arguments the differences are manifest in connection with fixed points.
Expansivity is a classical notion due to Bowen and Walters; it implies that (without loss of generality) there are no fixed points---they are isolated points of \(X\) \cite[Remark 1.7.3]{FH}. Kinematic expansiveness deals naturally with the presence of fixed points. Being separating suffices by itself for several arguments, and for others it does so if one also assumes the absence of fixed points. 

Kinematic expansivity is not invariant under orbit-equivalence or time-changes \cite[Tables 1, 2]{Artigue}; the property of a flow that all time-changes are kinematic-expansive is strong kinematic expansivity (likewise with ``separating''). Even this is more general than Bowen--Walters expansivity.
\end{remark}
\begin{proposition}[\cite{Gura2}]\label{PRPGura}
The horocycle flow of a negatively curved surface is strongly separating (but not expansive).
\end{proposition}
\begin{proposition}\label{PRPSeparationFinitary}
Suppose \(\Phi\) is a flow on a compact metric space \(X\).
\begin{enumerate}
\item\label{itemfiniteFP}If \(\Phi\) is separating, then fixed points are isolated, hence finite in number, and \label{itemNoSmallPeriods}\(\Phi\) does not have arbitrarily small positive periods.
\item\label{itemCloseTracking}If \(\Phi\) is kinematic-expansive and $\delta$ a separation constant for $\epsilon>0$, then for any $\rho>0$ there is a $T>0$ with
\[
d(\varphi^t(x),\varphi^t(y))<\delta\text{ for all }t\in[-T,T]\Rightarrow
d(y,\varphi^t(x))<\rho\text{ for some }t\in[-\epsilon,\epsilon];
\]
\end{enumerate}
\end{proposition}
\begin{proof}
\eqref{itemNoSmallPeriods}
For \(\delta\) as in the definition let \(\eta>0\) be such that \(d_{C^0}(\varphi^t,\textrm{id})<\delta\) for \(|t|<\eta\). If \(\Phi\) has arbitrarily small periods or infinitely many fixed points, then there are points \(x_n\) with periods \(p_n\to0\) which, without loss of generality, converge to a fixed point \(x\), which then has a \(\delta\)-neighborhood that contains a point \(y\) with positive period or a fixed point \(y\neq x\). Since \(d(\varphi^t(x),\varphi^t(y)<\delta\) for all \(t\in\R\), this contradicts being separating.

\noindent\eqref{itemCloseTracking}
Otherwise, take $x_n,y_n\in X$ such that $d(y_n,\varphi^t(x_n))>\rho$ for
all\/ $t\in[-\epsilon,\epsilon]$ and $d(\varphi^t(x_n),\varphi^t(y_n))<\delta$
for all\/ $t\in[-n,n]$, and (without loss of generality) $x_n\to x$ and
$y_n\to y$. Then on one hand, $\varphi^t(x)\neq y$ when $|t|\le\epsilon$,
while on the other hand for any $r\in\R$ we have
$d(\varphi^r(x_n),\varphi^r(y_n))<\delta$ for all\/ $n\ge K\dfn|r|$, so
$d(\varphi^r(x),\varphi^r(y))<\delta$, hence, since $r$ was arbitrary,
$y=\varphi^t(x)$ for some $t\in[-\epsilon,\epsilon]$, a contradiction.
\end{proof}
\begin{definition}[Hyperbolic set, Axiom A]\label{defhypsetforflow}
Let $M$ be a smooth manifold and \(\Phi\) a smooth flow on $M$.  A compact \(\Phi\)-invariant set $\Lambda$ is a \emph{hyperbolic set} for \(\Phi\) if there are a finite number of hyperbolic fixed points $\{p_1,..., p_k\}$, a closed set $\Lambda'$ such that $\Lambda=\Lambda'\cup\{p_1,..., p_k\}$, a \(\Phi\)-invariant splitting $T_{\Lambda'} M=E^s\oplus E^c\oplus E^u$, and constants $C\geq 1$, $\lambda\in(0,1)$ such that
\begin{itemize}
\item $E^c(x)\dfn\R X(x)\neq\{0\}$ for all\/ $x\in\Lambda'$, where $X\dfn\frac{d}{dt}\varphi^t(x)|_{t=0}$,
\item $\| D\varphi^t\rest{E^s_x}\|\leq C\lambda^t$ for all\/ $t>0$ and all\/ $x\in\Lambda'$, and
\item $\| D\varphi^{-t}\rest{E^u_x}\|\leq C\lambda^t$ for all\/ $t>0$ and all\/ $x\in\Lambda'$.
\end{itemize}
A smooth flow \(\Phi\) on a connected manifold $M$ is said to be an \emph{Anosov flow}
(or \emph{hyperbolic flow}) if \(M\) is a hyperbolic set for \(\Phi\).

A flow \(\Phi\) satisfies \emph{Axiom A} if\/ $\NW(\Phi)$ is hyperbolic and is the closure of the periodic orbits\rlap.\footnote{Our Axiom A allows for hyperbolic fixed points, whereas Smale's original Axiom A excluded singularities (he used ``{Axiom A'\,}'' for our Axiom A). Our choice follows Bowen's terminology.}

A hyperbolic set $\Lambda$ for $\Phi$ is said to be \emph{locally
maximal} if there is a neighborhood $V$ of\/ $\Lambda$ (an \emph{isolating neighborhood}) such that $\Lambda=\Lambda^V_\Phi\dfn\bigcap_{t\in\R}\varphi^t(V)$. A locally maximal hyperbolic set $\Lambda$ for a flow $\Phi$ is a \emph{basic set} if there is a positive semiorbit that is dense in $\Lambda$, that is, \(\Phi\rest\Lambda\) is topologically transitive.
\end{definition}
If a flow \(\Phi\) satisfies \emph{Axiom A}, then $\NW(\Phi\rest{\NW(\Phi)})=\NW(\Phi)$.
If\/ $\Lambda$ is a basic set, then $\NW(\Phi\rest{\Lambda})=\Lambda$.  The nonwandering set of an Axiom A flow is a finite union of disjoint basic sets by Smale's Spectral Decomposition Theorem \ref{THMSpecDecomp}.

The \emph{local strong stable manifold} and \emph{local
strong unstable manifold} of \(x\) are characterized as follows:
there is a continuous family of neighborhoods $U_x$ of\/ $x\in\Lambda$ such that
\begin{align*}
W^s_{\mathrm{loc}}(x)&=\{y\,\mid\, \varphi^t(y)\in U_{\varphi^t(x)}\textrm{ for }t>0, & d(\varphi^t(x),\varphi^t(y))\lto{t\to+\infty} 0\}, \\
W^u_{\mathrm{loc}}(x)&=\{y\,\mid\, \varphi^{-t}(y)\in U_{\varphi^{-t}(x)}\textrm{ for }t>0, & d(\varphi^{-t}(x),\varphi^{-t}(y))\lto{t\to+\infty} 0\}.
\end{align*}
The global \emph{strong stable} and \emph{strong unstable} manifolds
\begin{equation}\label{eqstrongstunstmfsforflows}
\begin{aligned}
W^s(x)&\dfn\bigcup_{t>0}\varphi^{-t}(W^s_{\mathrm{loc}}(\varphi^t(x)))=\{y\in M\,\mid\, d(\varphi^t(x),\varphi^t(y))\lto{t\to\infty} 0\},
\\
W^u(x)&\dfn \bigcup_{t>0}\varphi^t(W^u_{\mathrm{loc}}(\varphi^{-t}(x)))=\{y\in M\,\mid\, d(\varphi^{-t}(x),\varphi^{-t}(y))\lto{t\to\infty} 0\}
\end{aligned}
\end{equation}
are smoothly injectively immersed manifolds, as are
the manifolds
\begin{equation}
W^{cs}(x)\dfn \bigcup_{t\in\R}\varphi^t(W^s(x))
\text{ and }
W^{cu}(x)\dfn \bigcup_{t\in\R}\varphi^t(W^u(x)),
\label{eqweakstunstmfsforflows}
\end{equation}
called the \emph{weak stable} and \emph{weak unstable}
manifolds (or \emph{center-stable} and
\emph{center-unstable} manifolds) of
\(x\).  Note that $T_x W^{cs}=E_x^s\oplus E_x^c$,
$T_x W^{cu}=E_x^c\oplus E_x^u$.
\begin{theorem}[{In-Phase Theorem \cite[Theorem 5.3.25]{FH}}]\label{THMInPhase}
If\/ $\Lambda$ is a compact locally maximal hyperbolic
set for \(\Phi\) on $M$, then
\(\displaystyle
W^s(\Lambda)\dfn\{x\in M\mid\emptyset\neq\omega(x)\subset\Lambda\}=\bigcup_{x\in\Lambda}W^s(x)\), and
\(\displaystyle W^u(\Lambda)\dfn\{x\in M\mid\emptyset\neq\alpha(x)\subset\Lambda\}=\bigcup_{x\in\Lambda}W^u(x)\),
and for any $\epsilon>0$ there is a neighborhood $U$ of\/ $\Lambda$ with
$\displaystyle\bigcap_{t\ge0}\varphi^{-t}(U)\subset W^s_\epsilon(\Lambda)\dfn\bigcup_{x\in\Lambda}W^s_\epsilon(x)$ (and likewise for $W^u$).
\end{theorem}

If \(\Phi\) is an Axiom A flow, then $M=\bigcup_{i=1}^m W^s(\Lambda_i)=\bigcup_{i=1}^m W^u(\Lambda_i)$ with each union disjoint, where $\{\Lambda_i\}_{i=1}^k$ is the spectral decomposition.  Furthermore, there is an open and dense set of points that are contained in the basin of an attractor and a repeller.
\begin{definition}
If \(\Phi\) is an Axiom A flow, define a partial
ordering $\gg$ on the basic sets $\Lambda_1, \dots, \Lambda_n$ from the
spectral decomposition  by
\[\Lambda_i\gg \Lambda_j\text{ if }
(W^u(\Lambda_i)\smallsetminus\Lambda_i)\cap (W^s(\Lambda_j)\smallsetminus\Lambda_j)\neq \emptyset.\]
A \emph{$k$-cycle} consists of a sequence of basic sets $\Lambda_{i_1}\gg \Lambda_{i_2}\gg \cdots \gg \Lambda_{i_k}\gg \Lambda_{i_1}$.  The flow \(\Phi\)
has the \emph{no cycles property} if there are no cycles among the basic sets.
\end{definition}
The no cycles property for an Axiom A flow is equivalent to hyperbolicity of the chain recurrent set \cite[Theorem 5.3.35]{FH} and implied by the \emph{strong transversality} condition for an Axiom A flow assumed in \cite{PY89}:
$W^s(x)$ and $W^u(x)$ are transverse for all\/ $x\in M$.

\begin{definition}\label{defchains}
An \emph{\(\epsilon\)-chain} for a flow \(\Phi\) on a space \(X\) is a map
\(g\colon I\to X\) on a nontrivial interval\/ \(I\subset\R\) such that
\[
d(g(t+\tau),\varphi^{\tau}(g(t))) <\epsilon,\quad\text{for }t,t+\tau\in I\text{ and }|\tau| < 1.
\]
A point \(x\in X\) is \emph{chain recurrent} if it lies on periodic \(\epsilon\)-chains for every \(\epsilon>0\) (the set \(\mathcal{R}(\Phi)\) of such points is the chain-recurrent set), and chain-recurrent points \(x,y\) are \emph{chain-equivalent} if the pair lies on  periodic \(\epsilon\)-chains for every \(\epsilon>0\). The equivalence classes are called the \emph{chain-components}.
\end{definition}
The Anosov Shadowing Theorem \cite[Theorem 5.4.1]{FH} implies in particular the \emph{shadowing property} (\(\epsilon\)-chains are close to orbits)\footnote{This is also known as the pseudo-orbit tracing property.} as well as
the next result, that hyperbolic dynamics topologically do  not change under perturbation.  We will, however,  see that since derivatives can change under perturbations, so can the centralizer.
\begin{theorem}[Strong structural stability of hyperbolic sets]\label{thmstrucstabilityhypsets}Suppose $\Lambda$ is a compact hyperbolic set for a $C^1$ flow
\(\Phi\) on $M$. Then there are
\begin{itemize}
\item a $C^1$-neighborhood $U$ of \(\Phi\),
\item a $C^0$-neighborhood $V$ of the inclusion $\iota\dfn\mathrm{Id}\rest\Lambda$ of\/ $\Lambda$ in
$M$ and
\item a continuous map $h\colon U\to C(\Lambda,M)$, $\Psi\mapsto h_\Psi$
\end{itemize}
such that $h_\Phi=\iota$ and for each $\Psi\in U$
\begin{itemize}
\item$h_\Psi$ is a (H\"older) continuous embedding,
\item$h_\Psi$ is the transversely unique map in $V$ for which
$\psi^{\tau(t)}\circ h_\Psi=h_\Psi\circ \varphi^t\rest\Lambda$, where $\tau$ is given by the Shadowing Theorem,
\item the \emph{continuation} $\Lambda_\Psi\dfn h_\Psi(\Lambda)$ is a hyperbolic set for $\Psi$.
\end{itemize}
\end{theorem}
If \(\Psi=\Phi\) is Anosov, then transverse uniqueness comes close to establishing the  ``well-known and elementary'' triviality of centralizers of Anosov flows.

The shadowing property together with expansivity gives
\begin{theorem}[Spectral Decomposition]\label{THMSpecDecomp}
The chain-recurrent set of an expansive flow with the shadowing property has finitely many chain-components, and each is topologically transitive.
\end{theorem}

\section{Centralizers for kinematic-expansive flows}
In this section we prove Theorem \ref{t.expansive} 
about centralizers for kinematic-expansive flows.
Walters observed that expansive homeomorphisms have discrete
centralizers \cite[Theorem 2]{Walters}, and
likewise, kinematic expansivity of a flow (Definition \ref{d.expansive}) ensures that centralizers are discrete---provided we add hypotheses to control ``longitudinal'' phenomena. The Walters argument shows that a commuting homeomorphism preserves orbits, but unlike in discrete time we need to further establish that the shift along them is constant:
\begin{proposition}\label{PROPCentralizerCloseToIdentity}
Consider a continuous flow \(\Phi\) on \(X\) and \(f\in\mathcal{C}^0(\Phi)\).
\begin{enumerate}
\item\label{itemFixOrbits} If \(\Phi\) is kinematic-expansive, \(\epsilon>0\), \(\delta>0\) a separation constant for \(\epsilon\), \(d_{C^0}(f,\mathrm{Id})<\delta\), then \(\forall x\in X\ \exists t\in(-\epsilon,\epsilon)\) such that \(f(x)=\varphi^t(x)\).

\item\label{itemOrbitTranslation}\(f(x)\in\Oc(x)\Rightarrow\exists\tau=\tau(\mathcal{O}(x))\in\R\colon f\rest{\overline{\Oc(x)}}=\varphi^\tau\rest{\overline{\Oc(x)}}\).
\item\label{itemCtsShift}In the context of \eqref{itemFixOrbits}, \(x\mapsto\tau(\Oc(x))\) from \eqref{itemOrbitTranslation} can be chosen continuously on \(X\).

\item\label{itemTimeEpsNearId}If \(\Psi\) is a continuous flow on \(X\), then \(\forall\epsilon>0\ \exists\delta_0>0\colon|\delta|<\delta_0\Rightarrow d_{C^0}(\psi^\delta,\mathrm{Id})<\epsilon\).
\item\label{itemAczel}If 2 flows are the same \emph{set} of maps, then they are the same \emph{group} of maps: If \(\Phi,\Psi\) are continuous flows with \(\big\{\psi^t\mid t\in\R\big\}=\big\{\varphi^s\mid s\in\R\big\}\), and  \(\forall\epsilon>0\ \exists\delta>0\colon d_{C^0}(\varphi^t,\mathrm{id})<\delta\Rightarrow|t|<\epsilon>0\), then \(\exists c\in\R\ \forall t\in\R\quad\psi^t=\varphi^{ct}\).
\end{enumerate}
\end{proposition}
\begin{proof}
\eqref{itemFixOrbits}: Otherwise,  $d_{C^0}(f, \mathrm{Id})\ge\delta$ because there are $x\in X$ and $t\in\R$ with $\delta\le d(\varphi^t(f(x)), \varphi^t(x))=d(f(\varphi^t(x)), \varphi^t(x))$\rlap.\footnote{If \(\Phi\) is separating, then this argument shows that \(f(x)\in\varphi^\R(x)\) for all \(x\); absent fixed points, \cite[Proposition 1.1.12]{FH} likely gives \(|t|<\epsilon\).}

\eqref{itemOrbitTranslation}: Writing \(f(x)=\varphi^\tau(x)\) gives \(f(\varphi^t(x))=\varphi^t(f(x))=\varphi^{t+\tau}(x)=\varphi^\tau(\varphi^t(x))\), so \(f\rest{{\Oc(x)}}=\varphi^\tau\rest{{\Oc(x)}}\), and the claim follows by continuity of \(f\) and \(\varphi^\tau\).

\eqref{itemCtsShift}: For \(\epsilon>0\) less than the least positive period of \(\Phi\), \(\tau\in(-\epsilon,\epsilon)\) is uniquely determined for each \(x\in X\) that is not a fixed point. 
It extends continuously to \(X\) because it is uniformly continuous as follows. For \(\epsilon>0\) and \(\delta\) a separation constant for \(\epsilon\) take a suitable  \(\rho>0\) and \(T\) as in Proposition \ref{PRPSeparationFinitary}. If \(\eta>0\) is such that \(d(x,y)<\eta\Rightarrow d(\varphi^t(x),\varphi^t(y))<\delta\) when \(|t|\le T\), then \(d(y,\varphi^t(x))<\rho\) for some \(t\in[-\epsilon,\epsilon]\) and hence \(|\tau(x)-\tau(y)|\le\epsilon\)\rlap.\footnote{If \(\Phi\) has no fixed points, one does not need uniform continuity and can instead use separation to argue as follows: For \(\epsilon>0\) sufficiently small, \(\tau\in(-\epsilon,\epsilon)\) is uniquely determined for each \(x\in X\). If \(x_n\to x\), then \(n\mapsto\tau(x_n)\) has an accumulation point \(\tau_0=\lim_{k\to\infty}\tau(x_{n_k})\), and continuity of \(\Phi\) and \(f\) gives \(\varphi^{\tau_0}(x)=\lim_{k\to\infty}\varphi^{\tau(x_{n_k})}(x_{n_k})=\lim_{k\to\infty}f(x_{n_k})=f(x)\), so  \(\tau(x)=\tau_0=\lim_{n\to\infty}\tau(x_n)\).}

\eqref{itemTimeEpsNearId} is uniform continuity (in \(t\))  of \(\Psi\rest{[0,1]\times X}\).

\eqref{itemAczel}: The assumptions on \(\Phi\) imply that \(t\mapsto\tau(t)\) is well-defined on \(\R\) by \(\psi^t=\varphi^{\tau(t)}\) (since \(t\mapsto\varphi^t\) is injective) and continuous at 0. Then \(\varphi^{\tau(t+s)}=\psi^{t+s}=\psi^t\circ\psi^s=\varphi^{\tau(t)}\circ\varphi^{\tau(s)}=\varphi^{\tau(t)+\tau(s)}\) gives \(\tau(t+s)=\tau(t)+\tau(s)\) for all \(s,t\in\R\). This implies that \(\tau\) is continuous and then that \(\tau(t)\equiv ct\) for some \(c\in\R\) by \cite[\S 2.1: a measurable \(f\colon\R\to\R\) with \(f(t+s)=f(t)+f(s)\) is linear]{Aczel}.
\end{proof}

This gives Theorem \ref{t.expansive} by taking \(\Oc(x)\) dense in Proposition \ref{PROPCentralizerCloseToIdentity}\eqref{itemOrbitTranslation}. In fact, we proved:

\begin{proposition}[Discrete centralizer]
If \(\Phi\) is a topologically transitive kinematic-expansive continuous flow on \(X\), \(\epsilon\) a separation constant, \(f\in\mathcal{C}^0(\Phi)\), and \(d_{C^0}(f,\mathrm{Id})<\epsilon\), then \(f=\varphi^\tau\) for some \(\tau\).
\end{proposition}
Theorem \ref{t.expansive} does not properly generalize the corresponding
statement for Anosov flows because Anosov flows need not be transitive. The
following sidesteps that hypothesis:

\begin{proposition}\label{PRPTrChainCompNoConstOfMotion}
A continuous function invariant under a flow on a connected space \(X\)
with countably many chain-components, all of which are topologically
transitive, is constant.
\end{proposition}
We say that a flow \(\Phi\) \emph{has no constant of motion} if every continuous \(\Phi\)-invariant function is constant. Apparently, Thom conjectured that this is \(C^1\)-generic \cite[\S1]{Obata}.
\begin{proof}
Such a function \(h\) is constant on orbit closures of the flow \(\Phi\), hence\begin{itemize}
\item\(h\) is constant on each chain-component, and
\item if \(x\in X\), then \(h(\{x\})=h\big(\,\overline{\varphi^{\R}(x)}\,\big)=h(\omega(x))\) (so \(h(X)=h(L(\Phi))\)), and \(\omega(x)\) is contained in a chain-component \cite[Propositions 1.5.7(4), 1.5.34]{FH}.
\end{itemize}
Thus, \(h(X)\subset\R\) is countable and connected, hence a point.
\end{proof}
\begin{remark}
The transitivity assumption on chain-components is \(C^1\)-generically not needed, i.e., for a \(C^1\)-residual set of flows, every continuous invariant function is constant on each chain-component \cite[Lemma 6.17]{LOS}.
\end{remark}

Theorem \ref{THMExpansiveChainComp} follows from Proposition \ref{PRPTrChainCompNoConstOfMotion} and Proposition \ref{PROPCentralizerCloseToIdentity}\eqref{itemCtsShift}---as does the next result.

\begin{proposition}[Discrete centralizer]\label{PRPLocTrivCent}
A kinematic-expansive flow on a connected space with countably many
chain-components, each topologically transitive, has discrete \(C^0\)-centralizer (\(\exists\delta>0\colon f\in\mathcal{C}^0(\Phi)\), \(d_{C^0}(f,\mathrm{Id})<\delta\Rightarrow f=\varphi^\tau\) for some \(\tau\)).
\end{proposition}
By the Spectral Decomposition (Theorem \ref{THMSpecDecomp}) this gives in particular
\begin{theorem}
Expansive flows with the shadowing property have discrete centralizer.
\end{theorem}
\begin{remark}[Hayashi]
Variants of Theorem \ref{THMExpansiveChainComp} arise by showing that \(h\) in Proposition \ref{PRPTrChainCompNoConstOfMotion} is constant on chain-components \(C\) under hypotheses other than transitivity:
\begin{itemize}
\item The closing property---\(\forall\epsilon>0\exists\delta>0\colon\delta\)-chains are \(\epsilon\)-shadowed by a closed orbit; if \(x\sim y\) take \(x_i\to x\), \(y_i=\varphi^{t_i}(x_i)\to y\), hence \(h(x_i)=h(y_i)\)---but this argument also establishes transitivity.
\item The shadowing property---either by an analogous argument or because together with expansivity it implies the closing property---and hence transitivity in this context.
\item \(\exists x\in C\ \forall y\in C\ \exists x_i\to x,\ t_i\in\R\) with \(\varphi^{t_i}(x_i)\to y\).
\end{itemize}
\end{remark}
We expand on the preceding results and on the last of these suggestions by spelling out more carefully what these basic arguments establish. If \(\Phi\) has no constant of motion, then  \(T(\cdot)\) from Theorem \ref{THMWaltersCtsTime} is constant, so taking the conclusion of Theorem \ref{THMWaltersCtsTime} as the definition of having quasidiscrete centralizer implies
\begin{proposition}
A flow has discrete centralizer if it has quasidiscrete centralizer and no constant of motion.
\end{proposition}
Combined with Theorem \ref{THMWaltersCtsTime} itself, this gives
\begin{proposition}\label{PRPWaltersCtsTime}
If \(\Phi\) is a kinematic-expansive continuous flow on \(X\) that has no constant of motion, \(\epsilon>0\), \(\delta>0\) a separation constant for \(\epsilon\), \(f\circ\varphi^t=\varphi^t\circ f\), \(d_{C^0}(f,\mathrm{Id})<\delta\), then there is a \(T\) (near 0) such that $f=\varphi^T$.
\end{proposition}
\begin{remark}\label{REMConstOfMotion}
With the notations from Proposition \ref{PRPWaltersCtsTime} here are a few sufficient conditions for having no constant of motion; keep in mind that they are of interest here in the presence of kinematic expansivity (or of having quasidiscrete centralizer) because it is in that context that these then imply discrete centralizer.
\begin{itemize}
\item\(\Phi\) does not have the identity as a (nontrivial) topological factor. (A nonconstant invariant continuous function defines such a factor map.)
\item The limit set \(L(\Phi)\) (Definition \ref{DEFtransitive}) is contained in an at most countable union of \emph{elongational limit sets} \(\overline{\mathcal L}(x)\dfn\overline{\bigcup_{n\in\N}\mathcal{L}^n(x)}\), where  \(\mathcal L^n(x)\dfn\mathcal L(\mathcal L^{n-1}(x))\) and
\[
\mathcal{L}(A)\dfn\big\{\lim_{i\to\infty}\varphi^{t_i}(x_i)\mid\lim_{i\to\infty}x_i\in A,\ t_i\in\R\big\}=\bigcap\big\{\overline{\varphi^\R(O)}\mid A\subset O\text{ open}\big\},
\]
the \emph{elongation} of \(A\). (A constant of motion is constant on  \(\overline{\mathcal L}(x)\) and takes all its values on \(L(\Phi)\).)
\item More strongly, one can replace the elongational limit sets in the previous item by the \emph{elongational hulls} of points x, the closure of the smallest set containing \(x\) that is closed under application of \(\mathcal L\).
\end{itemize}
\end{remark}
One can contemplate what the nature of a kinematic-expansive flow with a constant of motion might be. The restriction \(\Phi_s\) of such a flow to a level set is itself kinematic-expansive. (We note that if the restriction to any level set is expansive \emph{and has the shadowing property}, then the flow is not expansive. Thus, any expansive such examples decompose into expansive flows none of which have the shadowing property. This illustrates how kinematic expansivity is a substantial generalization.) 

We previously remarked on uniqueness of conjugacies, and this is an interesting issue in this topological context because conjugacies are not often smooth. Thus (since \(k^{-1}h\) is in the centralizer of \(\Phi\) below), we note:
\begin{theorem}[Local uniqueness of conjugacies]\label{THMConjugacyUnique}
Suppose \(\Phi\) is continuous flow with discrete \(C^0\)-centralizer. If \(\Psi\) is topologically conjugate to \(\Phi\) via a homeomorphism \(h\), then \(h\) is locally unique, i.e.,  there is a \(\delta>0\) such that if \(k\) is a conjugacy between \(\Phi\) and \(\Psi\) with \(d_{C^0}(h,k)<\delta\), then \(h=k\circ\varphi^t\) for some small \(t\).
\end{theorem}
%
\begin{remark}
Our explorations of when a flow has no constant of motion are also pertinent to quasitriviality of the centralizer: the (diffeomorphism-) centralizer of a flow \(\Phi\) is said to be quasitrivial if it consists of maps of the form \(\varphi^{T(\cdot)}(\cdot)\); if the flow has no constant of motion, then \(T(\cdot)\) is necessarily constant and the centralizer is trivial. And this in turn then yields uniqueness (rather than local uniqueness) of conjugacies.
\end{remark}
Proposition \ref{PROPCentralizerCloseToIdentity} also implies in particular:
\begin{proposition}\label{LEMCentralizerCloseToIdentity}
Let \(\Phi\) be a \(C^r\) Axiom A flow on a closed manifold \(M\) and let $\epsilon>0$ be an expansive constant for $\Phi\rest{\NW(\Phi)}$. If\/ $f\in \mathcal{C}^0(M)$ and $d_0(f, \mathrm{Id})<\epsilon$,  then $f(x)\in \mathcal{O}(x)$ for all\/ $x\in\NW(\Phi)$.
\end{proposition}

This result points to two issues in identifying the centralizer. The first is having to deal with wandering points.  Hyperbolicity helps describe the centralizer of an Axiom A flow on the nonwandering set, but we will need perturbation methods to ``control'' centralizers on the wandering set.  The second is that the discreteness of the centralizer in Proposition \ref{LEMCentralizerCloseToIdentity}
helps show that any commuting \emph{flow} is a constant-time reparametrization, but for a diffeomorphism far from the identity much more is needed to show that it is a time-$t$ map of the flow. That is the substance of the next sections.

\section{Centralizers for Axiom-A flows}
We now prove Theorems \ref{t.periodictrivial}
and \ref{t.residualtrivial}.  We first outline the arguments.
The first step (Subsection \ref{SBSFixBasins}) ensures that the commuting diffeomorphism fixes the various attractors and repellers as well as their basins. This can often be established for commuting \emph{homeomorphisms} by considering periods of closed orbits, But since we work in the smooth category and want to utilize fixed points as well, we use that a closed orbit and its image under a commuting diffeomorphism must have conjugate derivatives (Lemma \ref{LEMCentralizerPeriodicPoints}), which allows us by perturbation to force the commuting diffeomorphism to fix a periodic orbit or fixed point in each attractor/repeller, and hence that whole set itself. It then clearly fixes the entire basins as well (Lemma \ref{LEMCentralizerStableManifolds}). Thus, there is an open and dense set $\mathcal{U}$ of flows such that any diffeomorphism $g$ commuting with a flow in $\mathcal{U}$ fixes the basins of each attractor or repeller, and each attractor or repeller contains at least one fixed or periodic point whose orbit is fixed by the commuting diffeomorphism $g$.

In Subsection \ref{SBSLocalToGlobal} we show that once a commuting diffeomorphism has been identified on an open subset of one of these basins, then it is globally identified. This is done in 2 parts. 
For the open and dense set $\mathcal{U}$, Theorem \ref{THMCentralizerAttractor} ensures  that if two commuting diffeomorphisms agree in an open set of a basin, then they agree for the entire basin, and Theorem \ref{THMopensetscentralizers} then links the basins of the attractors and repellers to let us conclude that there is an open and dense set $\mathcal{V}$ of flows such that if two commuting maps agree on an open set, then they agree on the entire manifold. This reduces the proof of Theorems \ref{t.periodictrivial} and \ref{t.residualtrivial} to a local problem analogous to the results in \cite{PY89} for maps:
it remains to show that on a basin of an attractor or repeller any commuting diffeomorphism is a time-$t$ map of the flow.

We previously mentioned that the heart of the problem is in controlling nonwandering points, and accordingly, this remaining portion of the proof is the most difficult. We carry it out in 2 parts.
Subsections \ref{ss.linearization} and \ref{ss.liegroup} explain the reduction to an algebraic problem.
More specifically, normal-forms theory allows us to translate the local problem to an algebraic one, and Lemma \ref{LEMpowerofflow} uses Theorems \ref{THMCentralizerAttractor} and \ref{THMopensetscentralizers} to establish that the solution of the algebraic problem does indeed imply the solution of the dynamical problem and hence Theorems \ref{t.periodictrivial} and \ref{t.residualtrivial}.
Finally, the perturbations to solve the algebraic problem are carried out in Section \ref{SBSNoCompactPart}.

\subsection{Fixing the basins}\label{SBSFixBasins}
The first step towards limiting what diffeomorphisms commute with a hyperbolic flow is to see that typically a commuting \emph{homeomorphism} fixes the ``large scale'' or ``combinatorial'' structure of the flow, namely the pieces of the chain decomposition (including the various attractors and repellers) and their respective basins. The latter is an easy consequence of the former, which suggests that this is a \(C^0\)-open circumstance.
\begin{lemma}\label{LEMCentralizerStableManifolds}
Let \(\Phi\) be a \(C^r\) Axiom A flow on a closed manifold \(M\), \(f\in\mathcal{C}^0(\Phi)\), and $x\in M$. Then
\[
f(W^s(x, \Phi))=W^s(f(x), \Phi)\textrm{ and }f(W^u(x, \Phi))=W^u(f(x), \Phi).
\]
\end{lemma}
The set of fixed points of a flow \(\Phi\) is invariant under any \(f\in\mathcal{C}^0(\Phi)\), as is the set of \(T\)-periodic \emph{orbits} for any \(T>0\) and, crucially, the period of each. Any  \(f\in\mathcal{C}^1(\Phi)\) furthermore conjugates the derivatives as follows.
\begin{lemma}\label{LEMCentralizerPeriodicPoints}
If \(f\in\mathcal{C}^0(\Phi)\), then the chain-recurrent set \(\mathcal{R}(\Phi)\) is \(f\)-invariant, and if\/ $p\in M$ is a fixed point or \(T\)-periodic point of \(\Phi\), then so is $f(p)$ (i.e.,  with the same period). If, furthermore, \(f\in\mathcal{C}^1(\Phi)\), then the derivatives of\/ $\varphi^T$ at $p$ and $f(p)$ are (linearly) conjugate.
\end{lemma}
\begin{proof}
If\/ $\varphi^t(p)=p$,  then $f(p)=f(\varphi^t(p))=\varphi^t(f(p))$.
If\/ $p\in M$ is a fixed point, then this holds for  all\/ $t\in \R$ and so $f(p)$ is a fixed point for \(\Phi\). If\/ $p\in M$ is \(T\)-periodic, then this holds for \(t=T\), so $f(p)$ is \(T\)-periodic. Differentiation of \(f(\varphi^t(p))=\varphi^t(f(p))\) then gives
\[
D\varphi^t(f(p))Df(p)=Df(\varphi^t(p))D\varphi^t(p)=Df(p)D\varphi^t(p).
\qedhere\]
\end{proof}
In particular, the spectrum of\/ $D\varphi^T(p)$ and $D\varphi^T(f(p))$ is the same; later this will be important for establishing triviality of the centralizer. However, \(C^0\)-arguments go some way: the spectral decomposition of an Axiom A flow \(\Phi\) is invariant under any \(f\in\mathcal{C}^0(\Phi)\) because \(f\) preserves chain-recurrence and chain-equivalence and hence permutes the chain-components of \(\Phi\).
\begin{lemma}\label{LEMPerturbToFixOrbits}
There is a \(C^0\)-open and \(C^\infty\)-dense set of Axiom A flows \(\Phi\) such that every \(f\in\mathcal{C}^0(\Phi)\) fixes each chain-component of \(\Phi\) that does not consist of a fixed point.
\end{lemma}
\begin{proof}
For each chain-component with a periodic orbit consider the least period in that chain component. These being pairwise distinct is a \(C^0\)-open condition and implies that these chain-components are each \(f\)-invariant (Lemma \ref{LEMCentralizerPeriodicPoints}), as are then their basins (Lemma \ref{LEMCentralizerStableManifolds}).

That these least periods are pairwise distinct is \(C^\infty\)-dense as follows: in each of these chain-components pick a periodic point \(p_i\) with that least period and a function \(\rho_i\) which is \(C^\infty\)-close to 1 and with \(\rho_i\equiv1\) outside a small neighborhood of \(p_i\) chosen such that the time-change \(\Phi'\) generated by the vector field \(\prod_i\rho_iX\), where \(X\) generates \(\Phi\), has distinct least periods.
\end{proof}
\begin{remark}\label{REMPerturbToFixOrbits}
Lemma \ref{LEMPerturbToFixOrbits} illustrates the presence of ``longitudinal'' effects specific to flows. This is related to the fact that \emph{conjugacies} between flows are rarer than orbit-equivalence, which is insensitive to time-changes, the very construction that gives rise to \(\Phi'\) in this proof. However, fixed points of \(\Phi\) have no meaningful longitudinal aspects, but \(C^1\)-techniques make sure they are fixed by \(f\): there is a \(C^1\)-open and \(C^\infty\)-dense set $\mathcal{U}_0$ of Axiom A flows such that each attractor or repeller has a fixed or periodic point where the derivative of the period or time-1 map is not conjugate to the corresponding derivative at any other such periodic orbit with the same period or fixed point. Lemma \ref{LEMCentralizerStableManifolds} then further implies that the basin of each attractor or repeller is fixed for any commuting diffeomorphism.
\end{remark}
In closing we note that we have not so far used the no cycles assumption.
\subsection{Rigidity: Local coincidence to global coincidence}\label{SBSLocalToGlobal}
The goal of this section is a global rigidity result, Theorem \ref{THMopensetscentralizers}, which may be of independent interest.

The first step, Theorem \ref{THMCentralizerAttractor}, is at the heart of reducing the proof of Theorems \ref{t.periodictrivial} and \ref{t.residualtrivial} to a local problem by fixing the commuting diffeomorphism on stable and unstable sets \emph{once it has been fixed on an open subset}; it is obtained by a minor modification of the discrete-time arguments in \cite{Anderson}. (For fixed or periodic attractors it is immediate from Theorem \ref{THMsimultaneouslinearization}.)
\begin{definition}
A linear map $A\colon\R^n\to\R^n$ is \emph{non\-resonant} if\/ $\mathrm{Re}\lambda_i\neq \mathrm{Re}\lambda_1^{m_1}\cdots \mathrm{Re}\lambda_n^{m_n}$ whenever \(0\le m_j\in\mathbb{Z}\) with $\sum m_j\geq 2$. A point \(p=\varphi^t(p)\) is non\-resonant if \(D\varphi^t(p)\) is.
\end{definition}

For a flow there is a similar, but slightly different notion of nonresonance that we use for the fixed points of the flow.

\begin{definition}
Denote the spectrum of an $n\times n$ matrix $A$ by $\Sigma(A)=\{\lambda_1,\dots, \lambda_n\}$, the eigenvalues of\/ $A$ repeated with multiplicity. \(A\) is said to be \emph{stable hyperbolic} if\/ $\mathrm{Re}\lambda<0$ for all\/ $\lambda\in \Sigma(A)$.
Define the function
\[
(\lambda, m)\mapsto\gamma(\lambda, m)\dfn\lambda-(m_1\lambda_1 + \cdots + \lambda_n m_n).
\]
A stable hyperbolic matrix is \emph{non\-resonant} if\/ $\mathrm{Re}\gamma(\lambda, m)\neq 0$ for any $m$ where $|m|\geq 2$ and any $\lambda\in \Sigma(A)$.
\end{definition}

There is an open and dense set $\mathcal{U}$ of \(\Phi\in \mathcal{U}_0\) for which each attractor contains a fixed or periodic point
that is nonresonant and where the derivative is not conjugate to the corresponding derivative at any other such fixed point or periodic orbit with the same period. Each attractor or repeller of such a flow then satisfies the hypotheses of the next theorem.
\begin{theorem}
\label{THMCentralizerAttractor}
Let \(\Phi\) be a $C^\infty$ flow on a manifold \(M\) and $\Lambda\subset M$ be a transitive hyperbolic attractor containing a fixed or periodic point $p$ that is non\-resonant.  If\/ $f_1, f_2\in\mathcal{C}^\infty(\Phi)$, and there exists an open set $V\subset W^s(\Lambda)$ such that $f_1\rest{V}=f_2\rest{V}$, then $f_1\rest{W^s(\Lambda)}=f_2\rest{W^s(\Lambda)}$.
\end{theorem}

We delay the proof of the result until the next section.
The desired global rigidity result, Theorem \ref{THMopensetscentralizers}, is now obtained by linking the basins of attractors and repellers.
In \cite{PY89} the maps are assumed to have strong transversality in order to link the basins, but after perturbations this can be done for Axiom A maps with the no cycles property \cite{Fis1}:
\begin{proposition}\label{prop.perturbationconnecting}
There is an open dense $\mathcal{V}\subset\mathcal{U}\subset\mathcal{A}^r(M)$ for $1\leq r\leq \infty$, such that if\/ $\Lambda$ and $\Lambda'$ are attractors for $\Phi\in \mathcal{V}$ with
\[
\overline{W^s(\Lambda)}\cap \overline{W^s(\Lambda')}\neq \emptyset,
\]
then there exists a hyperbolic repeller $\Lambda_r$, such that
\[
W^s(\Lambda)\cap W^u(\Lambda_r)\neq\emptyset\textrm{ and }W^s(\Lambda')\cap W^u(\Lambda_r)\neq\emptyset.
\]
\end{proposition}

The proof is almost identical to its discrete-time counterpart \cite[Proposition 3.2]{Fis1} as the proof uses properties of no cycles and perturbations on the wandering points, and these hold for flows as well.

This in turn implies Theorem \ref{THMopensetscentralizers}.

\begin{proof}[Proof of Theorem \ref{THMopensetscentralizers}]
Let $\Lambda_1,..., \Lambda_k$ denote the hyperbolic attractors of\/ $\Phi$.  If\/ $f\in  \mathcal{C}^\infty(\Phi)$ is the identity on a non\-empty open $U\subset M$, then there is an $i$ such that $\mathrm{int}(W^s(\Lambda_i)\cap U)\neq \emptyset$, hence $f\rest{W^s(\Lambda_i)}=\mathrm{Id}\rest{W^s(\Lambda_i)}$ by Theorem \ref{THMCentralizerAttractor}.
Now for $j$ such that
\[
\overline{W^s(\Lambda_i)}\cap \overline{W^s(\Lambda_j)}\neq \emptyset
\]
there is a repeller $\Lambda_r$ with
\[
W^s(\Lambda_i)\cap W^u(\Lambda_r)\neq\emptyset\neq W^s(\Lambda_j)\cap W^u(\Lambda_r).
\]
Therefore, $f$ is the identity on $W^u(\Lambda_r)\cup W^s(\Lambda_j)$ since the intersection of the basins is an open set.  Hence, $f$ is the identity on the open and dense set of points contained in the basin of an attractor or repeller.  Continuity of\/ $f$ implies that $f$ is the identity on all of\/ $M$.
\end{proof}

%

\subsection{Linearization theorems for flows and maps}\label{ss.linearization}
We now take the first step in the reduction to an algebraic problem. Under the nonresonance condition we have generically established, standard normal-form theory becomes the theory of smooth linearization \cite{Sell85} on the local stable manifold of a hyperbolic fixed or periodic sink or source or a periodic point for a hyperbolic attractor, and it implies that \emph{any element of the centralizer is simultaneously linearized.}


\begin{definition}\label{DEFSternbergcondition}
An $n\times n$ stable hyperbolic matrix $A$
satisfies the \emph{Sternberg condition of order} $N\geq 2$ if\/ $\mathrm{Re}\gamma(\lambda, m)\neq 0$ for all\/ $\lambda\in\Sigma(A)$ and $m=(m_1,\dots,m_n)\in\N^n$ with $|m|\dfn m_1 + \cdots + m_n\le N$.
\end{definition}
\begin{theorem}[Sternberg's Theorem]\label{THMsternberg}
Let $Q\geq 2$ and $R$ be $C^{2Q}$ on an open set $U\subset \R^n$ containing the origin.  If\/ $D^kR(0)=0$ for $k=0, 1$ and $A$ is a stable hyperbolic matrix (i.e.,  all its eigenvalues are inside the unit circle) such that $A$ satisfies the Sternberg condition of order $Q$, then the flow \(\Phi\) on \(\R^n\) generated by  $x'=Ax +R(x)$ admits a $C^{\lfloor Q/\rho\rfloor}$-linearization near 0, with \(\rho\) defined by
\[
\rho\dfn\rho(A)\dfn\frac{\max\{|\mathrm{Re}\lambda|\,:\, \lambda \in \Sigma(A)\}}{\min\{|\mathrm{Re}\lambda|\,:\, \lambda\in \Sigma(A)\}}.
\]
\end{theorem}
We remark that a similar result holds for an unstable hyperbolic matrix simply by taking the inverse of the flow.

We say that a stable hyperbolic matrix is \emph{non\-resonant} if\/ $\mathrm{Re}\gamma(\lambda, m)\neq 0$ for any $m$ and any $\lambda\in \Sigma(A)$.
The following immediate corollary is the main application of Sternberg's Theorem.
\begin{corollary}\label{CORsternbergflows}
If\/ $f\in C^{\infty}$ and $x'=f(x)=Ax +R(x)$ where $A$ is a non\-resonant stable hyperbolic matrix, then there exists a $C^{\infty}$-smooth linearization.
\end{corollary}
This is more natural in the form of a local restatement.
\begin{theorem}\label{THMSternbergMapsLinearization}
If\/ $f\colon\R^n\to\R^n$ is a $C^\infty$ diffeomorphism and the origin is a hyperbolic sink for $f$ with $Df(0)$ non\-resonant, then there exists a $C^\infty$-smooth linearization of\/ $f$.
\end{theorem}

Next, the nonresonance assumption implies that the centralizer of a nonresonant linear system consists of linear maps \cite[Theorem 10.1.14]{FH}, and as a consequence, the smooth linearization from Theorem \ref{THMSternbergMapsLinearization} for the stable manifold of a sink simultaneously smoothly linearizes any smooth map in the centralizer.

\begin{theorem}\label{THMsimultaneouslinearization}
Let $A\colon\R^n\to \R^n$ be a non\-resonant stable hyperbolic matrix and $\Phi_A$ the linear flow generated by $A$.  If\/ $g$ is a $C^{\infty}$ homeomorphism such that $g\varphi^t_A=\varphi^t_A g$ for all\/ $t\in\R$, then $g$ is linear.
\end{theorem}
These results imply that if we consider either a hyperbolic sink or source, or a periodic point for a hyperbolic attractor of a $C^\infty$ flow, then not only is the stable manifold of such a point linearizable, but any element of the centralizer is simultaneously linearized.

\begin{proof}[Proof of Theorem \ref{THMCentralizerAttractor}]
We first suppose that $p$ is a non\-resonant stable hyperbolic fixed point for \(\Phi\).  Then there exists a neighborhood $U$ of\/ $p$ and $\Phi\rest{U}$ is linear with the appropriate smooth coordinate system by Corollary \ref{CORsternbergflows}.  Let $f=f_1\circ f_2^{-1}$.  Then $f$ is a $C^\infty$ homeomorphism of \(M\) and $f\varphi^t=\varphi^t f$ for all\/ $t\in\mathbb{R}$.  Hence, $f\rest{U}$ is linear in the same smooth coordinate system by Theorem \ref{THMsimultaneouslinearization}.

There exists some $t>0$ such that $\varphi^t(V)\cap U$ is an open set.  By hypothesis, $f\rest{\varphi^t(V)}$ is the identity.  So in the local coordinate system $f$ is a linear  diffeomorphism on $U$ that is the identity on a nonempty open subset of\/ $U$.  Hence, $f$ is the identity on $U$.  Now for any $y\in W^{s}(p)$ there exists some $t$ such that $\varphi^t(y)\in U$.  Then $f(y)=(\varphi^{-t}f\varphi^t)(y)=y$ and $f\rest{W^{s}(p)}$ is the identity.  Hence, $f_1=f_2$ on $W^{s}(p)$.

More generally, we let $p$ be a non\-resonant hyperbolic periodic point contained in $\Lambda$.  Since $W^{cs}(\mathcal{O}(p))$ is dense in $W^s(\Lambda)$ by the In-Phase Theorem \ref{THMInPhase} and by the Spectral Decomposition Theorem, there exists some $T_0>0$ such that $W^{s}(\varphi^{T_0}(p))\cap V$ contains an open set in $W^{s}(\varphi^{T_0}(p))$.  Let $\pi(p)$ be the period of\/ $p$.  As above we define $f=f_1\circ f_2^{-1}$.  Then $f$ is a $C^\infty$ homeomorphism of \(M\) and $f\varphi^t=\varphi^t f$ for all\/ $t\in\mathbb{R}$.

By Theorem \ref{THMSternbergMapsLinearization} there exists a neighborhood $U$ of\/ $\varphi^{T_0}(p)$ in $W^{s}(\varphi^{T_0}(p))$ and a smooth coordinate system such that $\varphi^{\pi(p)}$ is linear on $U$.  Then there is some $n\in\mathbb{N}$ such that $\varphi^{n\pi(p)}(V)$ contains an open set in $U$.  Then, as above, $f$ is the identity in $U$.  Hence, $f$ is the identity on $W^{s}(\varphi^{T_0})(p)$.

Now for $y\in W^{s}(\mathcal{O}(p))$ there exists some $t$ such that $\varphi^t(y)\in W^{s}(\varphi^{T_0}(p))$.  Then $f(y)=(\varphi^{-t}f\varphi^t)(y)=y$.  Since $W^{s}(\mathcal{O}(p))$ is dense in $W^s(\Lambda)$ and $f$ is the identity on $W^{s}(\mathcal{O}(p))$ we see that $f$ is the identity on $W^s(\Lambda)$ so $f_1=f_2$ on $W^s(\Lambda)$.
\end{proof}

\subsection{Lie group of commuting matrices}\label{ss.liegroup}
By simultaneous linearization we have reduced the problem to an algebraic one.
We now define the Lie group of matrices that commute with a linear contraction.  In the case where $p$ is a periodic point the linearization is a linear contracting map and when $p$ is fixed there will be a linear contracting flow.


The parametric version of Sternberg's linearization \cite{Anderson} implies that the linearization depends continuously on the flow. (While stated for discrete time, the result holds for flows either by adapting the proof or by using a standard argument to show that the linearizing diffeomorphism for the time-1 map linearizes the flow \cite[Proof of Theorem 5.6.1]{FH}.)
\begin{remark}\label{REMAvsPsi}
For the open and dense set $\mathcal{V}\subset\mathcal{A}^\infty(M)$ from
Theorem \ref{THMopensetscentralizers}, $\Phi\in \mathcal{V}$, and the nonresonant fixed or periodic point $p=p(\Phi)$ 
whose orbit is fixed by any commuting diffeomorphism (Remark \ref{REMPerturbToFixOrbits}),
we have an embedding $\mathcal{E}(p,\Phi)\colon\R^n\to M$ with:
\begin{enumerate}
\item $\mathcal{E}(p,\Phi)(\R^n)=W^s(p,\Phi)=W^s(p)$;
\item $\Phi\mapsto\mathcal{E}(p,\Phi)$ is continuous (where we use the continuation of\/ $p$);\footnote{\(\Phi'\) near \(\Phi\) has a unique hyperbolic \(p'\) near \(p\) which varies continuously with \(\Phi'\).} and
\item\label{itemAvsPsi}for each connected component $\mathcal W$ of\/ $\mathcal V$ there exist $r, s\in \mathbb{N}$ with $r+2s=n$, coordinates $x_1,..., x_{r+s}\in \R^r\times \mathbb{C}^s$, and a continuous map $\lambda=(\lambda_1,..., \lambda_{r+s})\colon\mathcal{W}\to (\R^*)^r\times (\mathbb{C}^*\smallsetminus\R^*)^s$ such that
\begin{itemize}
\item if\/ $p$ is a hyperbolic fixed point, then $\Psi(p,\Phi)\dfn\mathcal{E}(p,\Phi)^{-1}\circ \varphi^t\rest{W^s(p)}\circ \mathcal{E}(p,\Phi)$ is a linear contracting flow on $\R^n$ that is diagonalized in the $x_i$-coordinates with eigenvalues $\lambda_i$ and depends continuously on $\Phi\in\mathcal{W}$;
\item if\/ $p$ is a hyperbolic periodic point with period $\pi(p)$, then $A(p,\Phi)\dfn\mathcal{E}(p,\Phi)^{-1}\circ \varphi^{\pi(p)}\rest{W^{s}(p)}\circ \mathcal{E}(p,\Phi)$
is a linear contracting map on $\R^n$ that is diagonalized in the $x_i$-coordinates with eigenvalues $\lambda_i$ and depends continuously on $\Phi\in\mathcal{W}$.
\end{itemize}
\end{enumerate}
\end{remark}

We now consider elements of the centralizer for these linearized maps. If\/ $h\in \mathrm{Diff}^\infty(\R^n)$ commutes with a linear nonresonant map, then $h$ is linear and furthermore diagonal with respect to the coordinates described above \cite{Kopell}.  We fix a neighborhood $\mathscr{V}$ of the matrix (either the matrix for the linearized map $A(p,\Phi)$ or the matrix for the linearized flow $\Psi(p,\Phi)$) such that for any $B\in\mathscr{V}$ the sign of the real and imaginary parts of the eigenvalues for $B$ agree with those of this matrix.

\begin{remark}[Linear centralizer group]\label{REMlinearcentralizergroup}
The set of invertible diagonal matrices that commute with an invertible diagonal matrix $A$
is an abelian group isomorphic to the disconnected Lie group $Z\dfn Z_{r,s}\dfn\R^{r+s}\times(\mathbb{Z}/2\mathbb{Z})^r\times(S^1)^s$ \cite{PY89}. The cyclic subgroup \(\langle \epsilon\rangle\)
generated by any
\[
\epsilon\in
\big\{(\theta_1,\dots,\theta_{r+s},\epsilon_1,\dots,\epsilon_{r+s})\in Z_{r,s}\ST\theta_i=1\ \forall i,\ \epsilon_j=1\ \forall j>r\big\}\sim(\Z/2\Z)^r
\]
is discrete in \(Z\), \[Z_0\dfn Z_{r,s,\epsilon}\dfn  Z_{r,s}/\langle\epsilon\rangle\] is a disconnected abelian Lie group, and \[Z_1\dfn Z'_{r,s,\epsilon}\dfn\ker\chi /\langle\epsilon\rangle\]
is the maximal compact subgroup of \(Z_0\),
where the surjective homomorphism $\chi$ from $Z_{r,s}$ to the hyperplane $\Sigma$ in ${\mathbb R}^{r+s}$ determined by $\sum_{i=1}^{r+s}\theta'_i = 0$, is defined by
\[\chi(\theta_1,\dots,\theta_{r+s},\epsilon_1,\dots,\epsilon_{r+s}) = (\theta_1-\theta_{\rm ave},\dots,\theta_{r+s}-\theta_{\rm ave})=(\theta'_1, \dots, \theta'_{r+s}).\]
Here, $\theta_{\rm ave}$ is the average value of\/ $\theta_1,\dots,\theta_{r+s}$.


For any $B\in{\mathscr V}$ (with ${\mathscr V}$ as defined after Remark \ref{REMAvsPsi} with \(r,s,\epsilon\) constant), the abstract group \(Z\) is naturally isomorphic to the centralizer \(Z(B)\)of \(B\), \(Z_0\) is isomorphic to \(Z(B)/\langle B\rangle\)
where \(\langle B\rangle\) is the cyclic group generated by \(B\), $Z_1$ represents rescalings\footnote{plus possibly some rotations (complex part) and flips (real part)} of \(\Phi\) (as opposed to rescalings of time) on \(W^s(\Oc(p))\) modulo period-maps,
and \(Z_0/Z_1\simeq\R^{r+s-1}\).
\end{remark}
\begin{remark}[Fundamental domain, orbit space]\label{REMFundDomSpOrbits}
As in \cite[Section 3.5, p.\ 87]{PY89},
we define \emph{fundamental domains} \(F\),
the \emph{spaces $S_A$ and $S_B$ of orbits} of\/ $A$ and $B$ in $\mathscr{V}$, and a canonical diffeomorphism of\/ $S_B$ onto $S_A$.  In continuous time the fundamental domains are homeomorphic to $S^{n-1}$, and in discrete time the sets are annuli.
\end{remark}
Each attractor (or repeller)
of\/ $\Phi\in\mathcal{U}_0$ 
(Remark \ref{REMPerturbToFixOrbits}) contains a fixed or periodic point $p$ such that if\/ $g\in\mathcal{C}^\infty(\Phi)$, then $g(\mathcal{O}(p))=\mathcal{O}(p)$.  Furthermore, each fixed or periodic point $p$ of \(\Phi\in\mathcal{V}\subset\mathcal{U}\) as in Theorem \ref{prop.perturbationconnecting} is nonresonant.

If\/ $p$ is $\pi(p)$-periodic, then 
there is a unique $\tau\in[0,\pi(p))$ such that $g'\dfn g\circ \varphi^\tau$ is the identity on $\mathcal{O}(p)$, and
\[
g'(W^s(p))=W^s(p), \quad g'(W^u(p))=W^u(p).
\]
Since $g'$ restricted to $W^s(p)$ is smooth, there is a linearization of\/ $g'$ restricted to $W^s(p)$, and we denote the corresponding element in $Z_0$ by $\bar{g}$.

Furthermore, if\/ $g'\rest{W^s(p)}=\mathrm{Id}$, then $g'\rest{W^s(\mathcal{O}(p))}=\mathrm{Id}$ and so $g'$ is the identity on an open set, $W^s(\mathcal{O}(p))$ is open,
and hence on $M$ (Theorem \ref{THMopensetscentralizers}), i.e.,  $g=\varphi^{-\tau}$.  So to complete our proof of Theorems  \ref{t.periodictrivial} and \ref{t.residualtrivial} we will perturb \(\Phi\) in a way that forces $g'\rest{W^s(p)}=\mathrm{Id}$.

\subsection{Perturbations near attractors}\label{SBSNoCompactPart}

We now adapt the perturbation techniques from \cite{PY89} to continuous time. We use the structural stability of hyperbolic attractors to make perturbations such that there is an open and dense set (or residual set depending on the situation) of flows such that no nontrivial element in the group $Z_0$ is in the centralizer for the perturbed system.
Lemma \ref{LEMpowerofflow} below shows how this implies these flows have a trivial centralizer.

In the proof of Theorem \ref{t.periodictrivial}, we use the attractor or repeller that is a fixed point or single periodic orbit and perturb the flow to first obtain triviality of the centralizer
 on the basin of attraction.  As described above this then extends to triviality of the centralizer on the entire manifold.   
 
 For elements in $Z_1$ the orbit of a point consists of points in a discrete group of invariant tori and this is the reason that $Z_1$ is referred to as the compact part; whereas, in $Z_0\smallsetminus Z_1$ we have orbits that tend toward infinity and approach certain eigendirections.  We then will work with each of these separately as the orbits have very different behavior.
 
One place where the different behavior of the orbits is seen is in the orbit space defined in Remark \ref{REMFundDomSpOrbits}. The proof of Theorem \ref{t.periodictrivial} uses the orbit space and carries out different perturbations for elements in $Z_1$ and those in $Z_0\smallsetminus Z_1$.
The first step is to show that after a perturbation there is no nontrivial element of the compact part, $Z_1$, that commutes with the linearized system.  We then perturb further so that no element from the noncompact part\rlap, $Z_0\smallsetminus Z_1$,  is in the centralizer for the linearized system.  The next definitions and comments further demonstrate why we divide the compact and noncompact components into different sections.

For a diagonal matrix $B\in {\mathscr V}$ whose diagonal entries are the eigenvalues $(\lambda_1,..., \lambda_{r+s})$ and a diagonal matrix $D$ whose diagonal elements are $(\mu_1,..., \mu_{r+s})$ we let, as in Remark \ref{REMlinearcentralizergroup},
$$\theta_i=\frac{\ln |\mu_i|}{\ln| \lambda_i|} \textrm{, }
\theta_i'=\theta_i-\frac{1}{ r+s}\sum_{j=1}^{r+s}\theta_j
\textrm{, and }
\epsilon_i=\frac{\mu_i}{\exp \theta_i \rho_i(B)},
$$
where $\exp \rho_i(B)=|\lambda_i|$ for $1\leq i\leq r$, and $\exp \rho_i(B)= \lambda_i$ for $r<i\leq r+s$.  We then have an isomorphism from diagonal matrices onto $Z_{r,s}$ given by $\Theta_B(D)=(\theta_1,..., \theta_{r+s}, \epsilon_1,..., \epsilon_{r+s})$.

For \(i\in\{1,..., r+s\}\) let
$$W_i\dfn\bigoplus_{j\in\{1,..., r+s\}}\begin{cases}\{0\}&j\neq i,\\\R&i=j.\end{cases}$$
For $\Gamma\subset\{1,..., r+s\}$, let \(W_\Gamma\dfn\bigoplus_{i\in \Gamma}W_i\), \(\Gamma^c\dfn\{1,..., r+s\}\smallsetminus \Gamma\),
and $\pi_\Gamma$ the projection with kernel\/ $W_{\Gamma^c}$ and image $W_\Gamma$.  For the map $B$ we let $\tilde{\pi}_\Gamma^B$ be the induced smooth map from $S_B\smallsetminus\widetilde W_{\Gamma^c}$ to $\widetilde W_\Gamma$.  Then the map $\tilde{\pi}_\Gamma^B$ commutes with the action of\/ $Z_0$ on $S_B$ \cite[Section 3.6]{PY89}, and $\widetilde W=S_B \smallsetminus \bigcup_\Gamma\widetilde W_\Gamma$,
where the union is over all proper subsets $\Gamma$ of\/ $\{1,\dots,r+s\}$.


In \cite[Section 3.3]{PY89} it is shown that $\Gamma(D)\dfn\big\{ i\in \{1,..., r+s\}\ST\theta_i'=\min \theta_j'\big\}$ satisfies
$\Gamma(D)=\{1,\dots,r+s\}$ if and only if\/ $D\in Z_1$. Thus, if\/ $D\in Z_0\smallsetminus Z_1$, then $\Gamma(D)$ is a proper nonempty subset of\/ $\{1,..., r+s\}$, so the projection $\pi_{\Gamma(D)}$ has nontrivial kernel, but also does not map all points to the origin. By contrast, in the compact case the projection $\pi_{\Gamma(D)}$ is the identity map, and this is not useful in the arguments that follow. This is one reason we deal with the compact case separately.

The proof of Theorem \ref{t.residualtrivial} where $\dim(M)=3$ is also split into two components.  The compact part is handled by modifying some of the arguments for the proof of Theorem \ref{t.periodictrivial}.  The noncompact part is handled very easily by the low dimensionality.  These combine to give us an open and dense set of flows with trivial centralizer.

 The proof of Theorem \ref{t.residualtrivial}, in arbitrary dimension, uses an attractor or repeller that is not a fixed point or single periodic orbit.  In fact, the construction can be done for any flow that has such an attractor or repeller, but in higher dimension we only obtain a residual set of flows with trivial centralizer.  The argument uses the homoclinic points of a periodic point contained in the attractor or repeller.  (Note that for an attractor or repeller that consists of a fixed point or single periodic orbit that there are no homoclinic points.)  Using the homoclinic points we do not need to separate the argument into the compact and noncompact part, since we do not need to use the orbit space
 to complete the proof.  The proof of Theorem \ref{t.residualtrivial}
 is therefore appended to the subsection with the considerations of the noncompact part.  The next lemma shows that the arguments described indeed prove Theorems \ref{t.periodictrivial} and \ref{t.residualtrivial}.
\begin{lemma}\label{LEMpowerofflow}
For $g\in \mathcal{C}^\infty(\Phi)$ we have $g=\varphi^t$ for some $t\in\R$ if and only if\/ $\bar g=1_{Z_0}$
\end{lemma}
\begin{proof}
If\/ $g=\varphi^t$ for some $t\in\R$ then $\bar g= 1_{Z_0}$ by definition of\/ $Z_0$.  Suppose that $\bar g=1_{Z_0}$.  Then by definition of\/ $Z_0$ we have
$g'=g\varphi^\tau=\varphi^s$ for some $s\in\R$ on $W^s(\mathcal{O}(p))$, so $g=\varphi^{s-\tau}$ on $W^s(\mathcal{O}(p))$.  By Theorems \ref{THMCentralizerAttractor} and \ref{THMopensetscentralizers} this implies that $g=\varphi^{s-\tau}$ on $M$.
\end{proof}

We now proceed to make the perturbations so each $g\in \mathcal{C}^\infty(\Phi)$ satisfies $\bar g= 1_{Z_0}$.

\subsubsection{Compact part of the centralizer}\label{SSnocompact}
In the proof of Theorem \ref{t.periodictrivial} the perturbation for the compact part of the centralizer, nontrivial elements in $Z_1$, is different than for the noncompact part as we described above.  
In this section we perturb the flow so that no nontrivial element in $Z_1$ commutes with it, and we show that the set of flows we obtain is open and dense in \(\mathcal{A}^\infty_1(M)\).
We will be working in the orbit space and examine the effect of the perturbations on the elements in the orbit space.\

We now show how to perturb the flows in the case of Theorem \ref{t.periodictrivial} to obtain an open and dense set of flows whose centralizer has no compact part.
Take $p\in M$ either a fixed hyperbolic sink or a periodic hyperbolic sink. We treat the following 3 cases in parallel.  In each case we will examine certain ``exceptional sets" or ``exceptional properties" and perturb the vector field generating the flow so that no element $g$ in the centralizer can have $\bar g\in Z_1\smallsetminus \{1_{Z_1}\}$ by examining the ``exceptional sets" or ``exceptional properties."   
\begin{enumerate}[label=\bf Case \arabic*.]
\item\label{itemcases} The basin of\/ $p$ is not contained in the basin of a single repeller.
\item The basin of\/ $p$ is contained in the basin of fixed or periodic source.
\item The basin of\/ $p$ is contained in the basin of a single repeller, which is not fixed or periodic.
\end{enumerate}

{\bfseries In Case 1}, let $J(p, \Phi)$ be the complement of the center unstable manifolds of the repellers in $W^s(p)$.  The set $J(p, \Phi)$ is a nonempty, nowhere dense, closed, flow invariant set in $W^s(p)$.  For the linearization $A$ of $W^s(p)$ we denote the orbit space of $W^s(p)$ as described in Remark \ref{REMFundDomSpOrbits} as $S_A=S(p)$.  Then $J(p, \Phi)$ projects to a nonempty, nowhere, dense, closed \(\Phi\)-invariant set $\tilde{J}(p, \Phi)\subset S(p)$.
Since any $g\in \mathcal{C}^\infty(\Phi)$ fixes the repellers and the unstable set of a repeller, it leaves $J(p, \Phi)$ invariant.  Likewise, the action induced by $g$ on the space of orbits leaves $\tilde{J}(p, \Phi)$ invariant.  Below we use the set $\tilde{J}(p, \Phi)$ to find a perturbation of the flow in such a way that $\tilde{J}(p, \Phi)$ is not invariant for any element of\/ $Z_1\smallsetminus\{1_{Z_1}\}$.

{\bfseries In Case 2}, let \(q\) be the fixed or periodic source with $W^s(p)\smallsetminus\{p\}\subset W^u(\mathcal{O}(q))\smallsetminus\{\mathcal{O}(q)\}$, $S(p)$
the orbit space for the fundamental domain of $W^s(p)$ (Remark \ref{REMFundDomSpOrbits}), and $S(q)$ the orbit space for the fundamental domain for $W^u(q)$.  Then each point in $S(p)$ corresponds to a unique point in $S(q)$, and an element of the Lie group that commutes with the linearization on $W^u(q)$ induces an action on $S(q)$ and hence an action on $S(p)$.  Below we will perturb the flow in a neighborhood of a fundamental domain of\/ $S(p)$ so that the perturbation does not change $S(q)$ and so that that no nontrivial element in the compact part can be in the centralizer.

{\bfseries In Case 3}, there is a foliation of\/ $W^s(p)\smallsetminus\{p\}$ by center-unstable manifolds of the repeller that is invariant and preserved by any element in $\mathcal{C}^\infty(\Phi)$.  There is then an invariant foliation of\/ $S(p)$ given by the image of the foliation under the projection to the space of orbits.  We let $\mathcal{F}(p, \Phi)$ be the leaves of the center-unstable foliation and $\tilde{\mathcal{F}}(p, \Phi)$ be the foliation on the space of orbits. Below we perturb the flow so that no nontrivial element in the compact part leaves the foliation invariant.

\begin{definition}\label{DEFZ1p}
We define the elements of\/ $Z_1$ that correspond to elements in $\mathcal{C}^\infty(\Phi)$ as follows, according to the three different cases on page \pageref{itemcases}. With notations as in Remark \ref{REMAvsPsi}\eqref{itemAvsPsi} set
$$B\dfn\begin{cases*}A&if\/ $p$ or $q$ is periodic,\\\Psi&if\/ $p$ or $q$ is fixed.\end{cases*}$$
\begin{enumerate}[label=\bf Case \arabic*:]
\item Let \(\displaystyle
Z_1(p,\Phi)=\{\bar{g}\in Z_1(B(p,\Phi))\ST g(\tilde{{J}}(p, \Phi))\subset \tilde{{J}}(p, \Phi)\}
.\)
\item Let 
\[
\strut\hskip-3pt Z_1(p,\Phi)\dfn
\{(\bar g, \bar{g}')\in Z_1(B(p, \Phi))\times Z_1'(B(q, \Phi))\ST\bar{g},\bar{g}'\textrm{ induce identical actions on }S(p)\}
.\]
\item Let \(\displaystyle
Z_1(p, \Phi)=\{\bar{g}\in Z_1(B(p, \Phi))\ST\forall x\in S, T_x\bar{g}(T_x \tilde{\mathcal{F}}(p, \Phi))=T_{\bar{g}(x)}\tilde{\mathcal{F}}(p, \Phi)\}
.\)
\end{enumerate}
\end{definition}

In each case, $Z_1(p, \Phi)$ is a closed subgroup of\/ $Z_1$.  Furthermore, for any closed subgroup $Z_2$ of\/ $Z_1$ the set
$$
\mathcal{V}_{Z_2}\dfn\{\Phi\in \mathcal{V}\ST Z_1(p, \Phi)\subset Z_2\}
$$
is open in $\mathcal{V}$ \cite[Lemma, p.\ 94]{PY89}.

We perturb the flow $\Phi$ to a new flow $\Phi'$ such that the two flows only differ in the interior of a fundamental domain and such that
$\Phi'=\Phi$ in a neighborhood of\/ $p$.  
Since the two flows agree in a neighborhood of the boundary of the fundamental domain, they both induce an action on $S(p)$ that is a diffeomorphism if\/ $p$ is periodic and a smooth flow if\/ $p$ is fixed.

To obtain the needed perturbations, note that $W\dfn\{(x_i)\in \R^r \times \mathbb{C}^s\ST x_i\neq 0\, \forall\,  i\}$ projects to an open dense set $\widetilde W\subset S$ which is invariant under each element of\/ $Z_0$. Since the axes correspond to the eigendirections for the matrices, $W$ consists of vectors that are not eigenvectors, so \(W\) is the set of orbits in the orbit space that are not orbits for eigenvectors. The reason to work in \(\widetilde W\) is that we can ensure that after a perturbation that any element $g$ of the centralizer such that $\bar g\in Z_1(p, \Phi)$ 
is the identity in $Z_1$; otherwise, we may have the identity in only some of the coordinates and would possibly need to make a series of perturbations.

We now show how a symmetry can be robustly broken:
\begin{lemma}\label{LEMnowheredense}
If\/ $g\in Z_1\smallsetminus\{1_{Z_1}\}$, then $\big\{\Phi\in \mathcal{U}\ST g\in Z_1(p, \Phi)\big\}$ is nowhere dense in $\mathcal{U}$.
\end{lemma}

\begin{proof}
{\bfseries Case 1}. Since $\widetilde W$ is open and dense in $S$, we may assume (by possibly passing to a small perturbation) that
$\tilde{{J}}(p, \Phi)\cap \widetilde W\neq \emptyset$.  Let $x\in \tilde{{J}}(p, \Phi)\cap \widetilde W$ such that $gx$ is not in the image of the boundary of the fundamental domain in $S$.  Since $\tilde{{J}}(p, \Phi)$ is nowhere dense we can make a perturbation \(\Psi\) which is the identity in the neighborhood of\/ $x$, but supported in a neighborhood of\/ $gx$.  To do this we modify $\tilde{J}(p, \Phi)$ so that it no longer contains $gx$;  so $gx\notin \tilde{{J}}(p, \Phi)=\tilde{\Psi}(\tilde{{J}}(p, \Phi))$.
So $g\notin Z_1(A(p, \Phi))$ for $p$ periodic or $g\notin Z_1(\Psi(p, \Phi))$ for $p$ fixed.

{\bfseries Case 2}. We can ensure that the fundamental domains of\/ $W^s(p)$ and $W^u(q)$ are disjoint and perturb the flow $\Phi$ so the support intersects the fundamental domain for $W^s(p)$, but not that of $W^u(q)$.  Then for $g_1\in Z_1(B(p, \Phi))$ and $g_2\in Z_1(B(q, \Phi))$ where $(g_1, g_2)\neq(1_{Z_1(B(p, \Phi))}, 1_{Z_1(B(q, \Phi))})$, we have $({g}_1, {g}_2)\in Z_1(p, \Phi)$ if the induced actions on $S(p)$ are the same.  For a point $x\in S(p)$ we can use a small perturbation $\Psi$ arbitrarily close to the identity such that ${g}_2 (x)\neq \tilde{\Psi}g_1\tilde{\Psi}^{-1}(x)$.  Then $({g}_1, {g}_2)\notin Z_1(p, \Phi')$, and this is an open condition.

{\bfseries Case 3}. The proof is similar to the proof for Case 1.  We let $x\in \widetilde W$ such that $gx$ is not in the boundary of the image of the fundamental domain in $S$.  We now perturb the flow $\Phi$ to a new flow $\Phi'$ such that the two flows agree in a neighborhood of the boundary of the fundamental domain and
$$
T_{gx}\tilde{{J}}(p,\Phi')
\neq T_{gx}\tilde{{J}}(p, \Phi)=T_x g(T_x\tilde{{J}}(p, \Phi)).$$
So $g\notin Z_1(A(p, \Phi))$ for $p$ periodic or $g\notin Z_1(\Psi(p, \Phi))$ for $p$ fixed.
\end{proof}

We now show how Lemma \ref{LEMnowheredense} implies that there is an open and dense set of flows such that there is no compact part of the centralizer.
For a closed subgroup $Z_2$ of\/ $Z_1$ we let
\[\mathcal{U}_{Z_2}=\{ \Phi\,:\, Z_1(\Phi)\subset Z_2\}.\]
The set $\mathcal{U}_{\{1_{Z_1}\}}$ is open (by the discussion before Lemma \ref{LEMnowheredense}) and dense in $\mathcal{U}$: Let $O$
be open in $\mathcal{U}$.  Then there exists some $\Phi\in O$ such that $Z_1(p, \Phi)$ is minimal among the $Z_1(p, \Phi')$ for $\Phi'\in O$.  Then for $\Phi'$ near $\Phi$ in $O$ we have $Z_1(p, \Phi')\subset Z_1(p, \Phi)$ \cite[Lemma 5.2]{PY89}.  By minimality this implies that $Z_1(p, \Phi')=Z_1(p, \Phi)$.  Lemma \ref{LEMnowheredense} now implies that $Z_1(p, \Phi)=\{1_{Z_1}\}$.  So there is an open and dense $\mathcal{U}_0\subset \mathcal{A}_1^\infty(M)$ such that for $\Phi\in \mathcal{U}_0$ and $g\in \mathcal{C}^\infty(\Phi)$ where $\bar g\in Z_1(p, \Phi)$ we have $\bar g= 1_{Z_1}$.

\subsubsection{Noncompact part of the centralizer}\label{SBSnononcompact}
We now show how to perturb the flows to eliminate the noncompact part of the centralizer.  We define $Z_0(p, \Phi)$ as we did for $Z_1(p, \Phi)$ in the three different cases (Definition \ref{DEFZ1p}). 

Theorem \ref{t.periodictrivial} follows from the next proposition.

\begin{proposition}\label{PROPtrivial}
$Z_0(p, \Phi)$ is trivial for each $\Phi$ in an open dense set $\mathcal{V}_1\subset\mathcal{V}\subset \mathcal{A}_1^\infty$.
\end{proposition}

\begin{proof}
As in the proof of Lemma \ref{LEMnowheredense} we divide this proof into the three different cases on page \pageref{itemcases}.  In each case we use an invariant closed exceptional set or exceptional properties.  Then we can choose a point $x\in S_A$ in the exceptional set and make a perturbation such that $\tilde{\pi}^A_\Gamma(x)$ is not in the exceptional set for any proper nonempty subset $\Gamma$ of\/ $\{1,.., r+s\}$. The next lemma then shows that if there is a nontrivial element in the centralizer there is some $\Gamma$ such that $\tilde{\pi}^A_\Gamma(x)$ is in the exceptional set, a contradiction.

\begin{lemma}[{\cite[Lemma 1]{PY89}}]\label{LEMaccumulation}
Let $A\in\mathcal{V}$ and $h\in Z_0\smallsetminus Z_1$.  Then
$$\lim_{n\to \infty} d(h^nx, h^n(\tilde{\pi}^A_{\Gamma(h)}(x)))=0$$ for any $x$ in $S_A\smallsetminus\widetilde W_{\Gamma(h)^c}$ and there is a subsequence with $\lim_{k\to \infty}h^{n_k}x=\tilde{\pi}^A_{\Gamma(h)}(x)$.
\end{lemma}

{\bfseries In Case 1}, say that $\Phi\in {\mathcal V}$ belongs to ${\mathcal V}_1$ if and only if there exists $x\in \tilde J(p,\Phi) \cap \widetilde W$ such that for any nontrivial proper subset $\Gamma$ of\/ $\{1,\dots,r+s\}$ the point $\tilde \pi_\Gamma^A \in \widetilde W_\Gamma$ does not belong to $\tilde J(p,\Phi)$.
The set ${\mathcal V}_1$ is open because $\tilde \pi_E^A$ and $\tilde J(p,\Phi)$ depend continuously on $\Phi \in {\mathcal V}$. For any $\Phi\in {\mathcal V}$, we choose a special small perturbation $\Phi^\prime$ such that for some $x\in \tilde J(p,\Phi')$ we have $\tilde \pi_\Gamma^A(x)\not\in \tilde J(p,\Phi')$ for any nonempty proper subset $\Gamma$ of\/ $\{1,\dots,r+s\}$.
To perturb the flow we can change the vector field we fix $\Gamma$ and a small  neighborhood of a point corresponding to $\tilde \pi _\Gamma^A(x)$ so that $x\in \tilde J(p, \Phi')$, but $\tilde \pi_\Gamma^A(x)\notin \tilde J (p, \Phi')$.  We do this for each $\Gamma$.
Then $\Phi^\prime$ belongs to ${\mathcal V}_1$, so ${\mathcal V}_1$ is dense.

For $\Phi\in {\mathcal V}_1$, $\bar g\in Z_0(p,\Phi) \smallsetminus Z_1(p,\Phi)$, and $x\in \tilde J(p,\Phi)\cap \widetilde W$ as in the definition of\/ ${\mathcal V}_1$, Lemma \ref{LEMaccumulation} implies $\tilde{\pi}^A_{\bar{g}}(x)\in \tilde{J}(p, \Phi)$, a contradiction. So $Z_0(p, \Phi)$ is trivial.

{\bfseries In Case 2}, we define $\widetilde W$ for $p$ as described above as an open and dense set of\/ $S_A$ and define $\widetilde W'$ similarly for the point $q$.  We say $\Phi\in \mathcal{V}$ belongs to $\mathcal{V}_1$ if there belongs some point $x\in \widetilde W\cap \widetilde W'$ such that for any proper nonempty set $\Gamma$ we have $\tilde{\pi}^A_\Gamma(x)\in \widetilde W'$.  Since the sets $\widetilde W$ and $\widetilde W'$ and the function $\tilde{\pi}^A$ depend continuously on $\Phi$ we see that $\mathcal{V}_1$ is open.  To see that it is dense again perturb the flow near a point in $\widetilde W\cap \widetilde W'$.
To do this let $x\in \widetilde W\cap \widetilde W'$
and perturb the flow to a flow $\Phi'$ such that the vector fields for the flows agree on a small neighborhood of the fundamental domain for $S_A$ and in a neighborhood of the orbit of $x$ and such that $\tilde{\pi}^A_\Gamma(x)\in \widetilde W'$ for any proper nonempty subset of\/ $\Gamma$.

For $\Phi\in \mathcal{V}_1$ and $(g_1,g_2)\in (Z_0\times Z_0')\smallsetminus (Z_1\times Z_1')$ such that $(g_1, g_2)$ have the same action on $S_A$ we have $g_1\in Z_0\smallsetminus Z_1$ and $g_2\in Z_0'\smallsetminus Z_1'$. By Lemma \ref{LEMaccumulation}, $\tilde{\pi}^A_{g_1}(x)$ is a limit point of\/ $\{g_1^n(x)\}_{n\geq 0}$ and so the limit set of\/ $\{g_2^n(x)\}_{n\geq 0}$ is contained in $S_A\smallsetminus \widetilde W'$.  This contradicts the choice of\/ $x$ for the flow.

{\bfseries In Case 3}, we let $\mathcal{V}_1$ be the set of\/ $\Phi\in \mathcal{V}$ such that for some point $x\in\widetilde W$ and any proper nonempty set $\Gamma$ we have $T_x \tilde{F}(p,\Phi)$ transverse to $W_{\Gamma^c}$ and $T_{\tilde{\pi}^A_\Gamma(x)}\tilde{F}(p, \Phi)$ is transverse to $W_\Gamma$.  The set $\mathcal{V}_1$ is open since the tangent spaces depend continuously and transversality is then an open condition.  To see that it is dense we perturb the flow to obtain the transversality conditions.

To see that  $\Phi\in \mathcal{V}_1$ and $g\in Z_0\smallsetminus Z_1$ does not leave $\tilde{F}(p, \Phi)$ invariant we use the next lemma to obtain the contradiction.

\begin{lemma}[{\cite[Lemma 2]{PY89}}]\label{LEMsubpsaces}
Let $\Phi\in \mathcal{V}$, $h\in Z_0\smallsetminus Z_1$, $x\in S_A\smallsetminus\widetilde W_{\Gamma(h)^c}$, and $\{n_k\}$ a sequence of integers satisfying the conclusion of Lemma \ref{LEMaccumulation}.  If\/ $V$ is a subspace of\/ $T_xS$ transverse to $\widetilde W_{\Gamma(h)^c}$ such that
$$
\lim_{k\to\infty}T_x h^{n_k}(V)=V_0\subset T_{\tilde{\pi}^A_h(x)}S_A,$$
then either $V_0\subset W_\Gamma(h)$ or $W_\Gamma(h)\subset V_0$.
\end{lemma}

This proves Proposition \ref{PROPtrivial} and hence Theorem \ref{t.periodictrivial}.
\end{proof}
\subsubsection{Proof of Theorem \ref{t.residualtrivial}}
Since the conclusion of Theorem \ref{t.periodictrivial} is stronger than the conclusion of \ref{t.residualtrivial} the theorem holds if there is an attractor or repeller that is a fixed point or single periodic orbit. We may then assume that all attractors and repellers for our flow are neither a single periodic orbit nor a fixed point.


Let $\mathcal{V}$ be an open dense set of flows from Theorem \ref{THMopensetscentralizers} 
and such that all attractors or repellers contain a periodic orbit that is nonresonant and such that the orbit is fixed by each element of the centralizer as in Lemma \ref{LEMPerturbToFixOrbits}.
Let $\Phi\in \mathcal{V}$ and $\Lambda$ be an attractor and $p\in \Lambda$ be a periodic point of\/ $\Lambda$ such that $g(p)\in \mathcal{O}(p)$ for each $g\in \mathcal{C}^\infty(\Phi)$.  The set of homoclinic points related to $p$ is $J(\Phi)=W^s(p)\cap W^{cu}(p)\smallsetminus\{p\}$.  For each $q\in J(\Phi)$ there is a unique point $q'\in W^u(p)$ such that $q'=\varphi^s(q)$ where $0\leq s < \pi(p)$ and $\pi(p)$ is the period of\/ $p$.  There exist linearizations of\/ $W^s(p)$ and $W^u(p)$ as described in Remark \ref{REMAvsPsi}.  If\/ $\dim(E^s)=n_1$ and $\dim(E^u)=n_2$, then there is a map $h$ from $J(\Phi)$ into $\mathbb{R}^{n_1}\times \mathbb{R}^{n_2}$ given by
\[
h(q)=(\mathcal{E}^s(p, \Phi)^{-1}(q), \mathcal{E}^u(p, \Phi)^{-1}(q'))
\]
where $\mathcal{E}^s(p, \Phi)$ is the linearization of\/ $W^s(p)$ and $\mathcal{E}^u(p, \Phi)$ is the linearization of\/ $W^u(p)$.  The map $h$ is injective and we let $\tilde{J}(p,\Phi)=h(J(p, \Phi))$.  This is a discrete closed set in $\mathbb{R}^{n_1}\times \mathbb{R}^{n_2}$.  Furthermore, the set $\tilde{J}(p, \Phi)$ is invariant under the transformation $A(\Phi)=(A_s, A_u^{-1})$ where $A_s$ is the linearization of the flow on $W^s(p)$, and $A_u$ is the linearization of the flow on $W^u(p)$.  Furthermore, for any element $g\in C^\infty(\Phi)$ and $\bar g_s$ the linearization for $g$ of $W^s(p)$ and $\bar g_u$ the linearization for $g$ of $W^u(p)$ the set $\tilde{J}(p, \Phi)$ is invariant for $(\bar g_s, \bar g_u^{-1})$.  The proof of Theorem \ref{t.residualtrivial} follows from this next proposition.  This proposition and its proof are almost identical and proof of Proposition 1 of \cite[p. 92, p. 95]{PY89}.

%
%
%

\begin{proposition}
If\/ $\dim(M)=3$ there is an open and dense set $\mathcal{U}$ of flows such that if $\Phi\in \mathcal{U}$ then no $g\in Z_0\smallsetminus \{1_{Z_0}\}$ leaves $\tilde{J}(p,\Phi)$ invariant.  If\/ $\dim(M)\geq 4$, then there is residual set of flows $\mathcal{R}\subset \mathcal{V}$ such that no $g\in Z_0\smallsetminus\{1_{Z_0}\}$ leaves $\tilde{J}(p, \Phi)$ invariant for $\Phi\in \mathcal{R}$.
\end{proposition}

\begin{proof}
Let $x, y\in J(p, \Phi)$ such that $y$ is not in the orbit of\/ $x$.  Let $h(x)=(x_1, x_2)$ and $h(y)=(y_1, y_2)$. Fix $x_1'$ sufficiently close to $x_1$ and select a small neighborhood $V$ of $\varphi^{\pi(p)}(x)$ so that it does not intersect the closed set consisting of $p$ the orbit of\/ $y$ and the backward orbit of $x$.  We now let $\Phi'$ be a perturbation of the flow so the the flows agree outside of\/ $V$ and the stable manifold of\/ $p$ in $V$ is the same for the two flows.  This can be done in such a way that $h(x)=(x_1', x_2)$ for the perturbed map.

Since the flows $\Phi'$ and $\Phi$ agree in a neighborhood of\/ $p$ the linearizations are the same.  Furthermore, it is not possible for some $g\in Z_0$ to satisfy $g(x)=y$ for both $\Phi$ and $\Phi'$.  Since the set $\tilde{J}(p, \Phi)$ is discrete we know that if\/ $\bar g\in Z_0\smallsetminus\{1_{Z_0}\}$ that the set of\/ $\Phi$ such that $\tilde{J}(p, \Phi)$ is $g$-invariant is nowhere dense.

Since the homoclinic points $J(p, \Phi)$ are countable there is a residual set $\mathcal{R}$ of \(\Phi\in \mathcal{V}\) such that if $g\in Z_0$ and $g(\tilde{J}(p, \Phi)=\tilde{J}(p, \Phi)$, then $g=1_{Z_0}$.

We now assume assume that $\dim(M)=3$.  
Modifying the above argument, together with the proof of Case 1 in Lemma \ref{LEMnowheredense}, and the argument just after the proof of Lemma \ref{LEMnowheredense} we see that there is an open and dense set $\mathcal{U}$ of flows such that if $g\in Z_1\setminus 1_{Z_1}$ and $\Phi\in\mathcal{U}$ that $g$ does not leave $\tilde{J}(p, \Phi)$ invariant.
In this case the stable and unstable manifolds are one dimensional and the transformation $A(\Phi)=(A_s, A_u^{-1})=(\lambda_1, \lambda_2)$.  Let $D$ be a diagonal matrix with diagonal entries $(\mu_1, \mu_2)$.  For the associated action of\/ $D$, denoted by $h$, we see that $h\in Z_0-Z_1$ if and only if\/ $$
\frac{\log|\mu_1|}{\log|\lambda_1|}\neq \frac{\log|\mu_2|}{\log|\lambda_2|}.$$
Then there exists some $k,l\in\mathbb{Z}$ such that $A(\Phi)^k h^l$ is a contraction---which contradicts $\tilde{J}(p, \Phi)$ being discrete and invariant. 
So each flow in in the open and dense set $\mathcal{U}$ has trivial\/ $\mathbb{R}$-centralizer.
\end{proof}

\ifdefined\href\renewcommand{\MR}[1]{\relax\ifhmode\unskip\space\fi MR\href{http://www.ams.org/mathscinet-getitem?mr=#1}{#1}}\fi
\bibliographystyle{amsplain}
\bibliography{ourbib-BFH2018}{}
\end{document}